\documentclass[a4paper,11pt]{article}

\usepackage{authblk}
\usepackage{a4wide}
\usepackage[utf8]{inputenc}
\usepackage{amsmath}
\usepackage{amsfonts}
\usepackage{amssymb}
\usepackage{mathrsfs}
\usepackage{amsthm}
\usepackage{color}
\usepackage{stmaryrd}
\usepackage[hidelinks]{hyperref}

\usepackage{latexsym,enumerate}
\usepackage{xcolor}
\usepackage{graphicx}
\usepackage{epstopdf}

\usepackage[T1]{fontenc}
\usepackage[utf8]{inputenc}
\usepackage[english]{babel}

\newcommand{\dx}[0]{\,\mathrm{d}}

\newcommand{\fbsde}[0]{}
\newcommand{\esssup}[0]{\mathrm{ess}\,\mathrm{sup}}

\newcommand*{\IR}{\mathbb{R}}

\newcommand*{\R}{\mathbb{R}}

\def \1{\mathbf{1}}

\newcommand{\be}{\begin{eqnarray*}}
\newcommand{\ee}{\end{eqnarray*}}
\newcommand{\ben}{\begin{eqnarray}}
\newcommand{\een}{\end{eqnarray}}
\newcommand{\bi}{\begin{itemize}}
\newcommand{\ei}{\end{itemize}}

\newcommand{\tdiv}{D}
\newcommand{\bmdiv}{\mathfrak{D}}

\definecolor{color2}{gray}{0.7}

\newtheorem{thm}{Theorem}[section]

\newtheorem{lemma}[thm]{Lemma}
\newtheorem{propo}[thm]{Proposition}
\newtheorem{corollary}[thm]{Corollary}

\theoremstyle{definition}

\newtheorem{defi}[thm]{Definition}
\newtheorem{definition}[thm]{Definition}

\newtheorem{remark}[thm]{Remark}

\title{The method of decoupling fields generalized to higher spatial derivatives}
\author{Alexander Fromm\thanks{alexander.fromm@uni-jena.de}}
\affil{\small Institute for Mathematics, University of Jena, Ernst-Abbe-Platz 2, 07743 Jena, Germany}

\begin{document}

\maketitle

\begin{abstract}
This work studies the spatial derivatives of decoupling fields to strongly coupled forward-backward stochastic differential equations in a Brownian setting. We formally deduce the backward dynamics of the first and higher spatial derivatives. In addition, we study necessary conditions under which singularities in either one of them can occur while moving backwards in time.
\end{abstract}

\vspace{0.5cm}
\noindent \textbf{2010 Mathematics Subject Classification.} 60H30, 49J55, 93E20.

\smallskip
\noindent \textbf{Keywords.} forward-backward stochastic differential equation, decoupling field, stochastic control, Skorokhod embedding

\section*{Introduction}

Forward-backward stochastic differential equations (FBSDE) are a class of differential problems which appear in numerous areas of applied stochastics. Most notably, stochastic control problems are often reduced to FBSDE. %The practical importance and applicability of this class of systems can be hardly underestimated and is additionally underlined by the fact that 
In the so-called Markovian case these systems can also be viewed as stochastic formulations of a large class of partial differential equations, covering many phenomena in physics, chemistry and engineering.

Of particular interest are so-called coupled systems in which neither the forward equation nor the backward equation, which together form the FBSDE, can be solved independently of the other. Although decoupled or weakly coupled problems appear rather often in the literature, it is a general pattern that the FBSDE associated with a particular control problem can be transformed into a decoupled system only under special structural properties and, in general, strongly coupled FBSDE cannot be avoided.
It is a longstanding challenge to find conditions guaranteeing that a given fully coupled FBSDE possesses a solution. Sufficient conditions are provided e.g.\ in \cite{Ma1994}, \cite{Pardoux1999}, \cite{ma:yon:99}, \cite{peng:wu:99}, \cite{Delarue2002}, \cite{ma:wu:zhang:15} (see also references therein). The method of decoupling fields, developped in \cite{Fromm2015} (see also the precursor articles \cite{ma:yin:zhan:12}, \cite{Fromm2013} and \cite{ma:wu:zhang:15}), is practically useful for determining whether a solution exists.
A decoupling field describes the functional dependence of the backward part $Y$ on the forward component $X$. If the coefficients of a fully coupled FBSDE satisfy a Lipschitz condition, then there exists a maximal non-vanishing interval possessing a solution triplet $(X,Y,Z)$ and a decoupling field with nice regularity properties. The method of decoupling fields consists in analyzing the dynamics of the decoupling field's gradient in order to determine whether the FBSDE has a solution on the whole time interval $[0, T]$.
The method can be successfully applied to various problems involving coupled FBSDE: In \cite{Proemel2015} solutions to a quadratic strongly coupled FBSDE with a two-dimensional forward equation are constructed to obtain solutions to the Skorokhod embedding problem for Gaussian processes with non-linear drift. In \cite{2017arXiv171106033F}
%Chapter 5 of \cite{Fromm2015} 
the problem of utility maximization in incomplete markets is treated for a general class of utility functions via construction of solutions to the associated coupled FBSDE. In \cite{ankirchner:hal-01500311}, the method is used to obtain solutions to the problem of optimal position targeting for general cost functionals and in \cite{ankirchner:hal-01615043}, the problem of optimal control of diffusion coefficients is treated using decoupling fields.

In all these applications the key step is to obtain the dynamics of the spatial derivative of the decoupling field evaluated along the forward process and to show that as a process it satisfies a backward SDE. In that it is similar to the backward process which is in fact the decoupling field itself evaluated along the forward process. It is natural to ask under what conditions higher order spatial derivatives of the decoupling field exist and whether they satisfy a backward equation as well. %More importantly, it is desirable to deduce the dynamics of these backward equations.
%, which is a key objective of this work.

Apart from scientific curiosity such study of higher spatial derivatives has significance from the point of view of various practical applications: One major motivation is that the FBSDE the decoupling field is constructed for is usually induced by an underlying problem the solvability of which at times depends on certain regularity properties of the solution to the FBSDE which in turn are reduced to regularity properties of the associated decoupling field. For instance, in \cite{Proemel2015} classical solutions to the FBSDE entail weak solutions to the Skorokhod embedding problem, while the construction of strong solutions requires differentiability of the decoupling field in time and space. This can be reduced to differentiability in space up to a sufficiently high order. However, showing spatial differentiability does become a bottle neck when generalizing the results of \cite{Proemel2015}, i.e. when constructing strong solutions to the Skorokhod embedding problem for more general diffusions, as a case by case study of higher derivatives becomes increasingly technical and complicated without a sufficiently general theoretical machinery specifically designed for this purpose. In other words, a goal is to streamline the arguments of section 4.2 of \cite{Proemel2015} and do so for a more general setting.

Probably the greatest motivation for studying spatial derivatives of decoupling fields and their dynamics comes from the subject of numerical analysis of FBSDE: In order to treat a forward-backward system numerically some discretization in time and space must occur to keep the processing time and capacity finite. This means in particular that for any given moment in time the decoupling field is to be evaluated or estimated at finitely many points only. In order to approximate the decoupling field at points not belonging to a given grid it is natural to apply some form of interpolation. In the one-dimensional case linear interpolation, even though limited in precision, seems natural, but is not straightforward to generalize in a higher-dimensional setting. Instead, it is more natural to approximate the field in a neighborhood of a given grid point using its Taylor expansion in space. For this Taylor expansion, however, derivatives of higher orders are needed.

A key finding of this work is that, fortunately, these higher order spatial derivatives of a decoupling field at a given point are not more complicated to obtain than the value of the decoupling field itself at the same point, since these spatial derivatives evaluated along the forward process do satisfy a backward SDE and, therefore, behave similarly to the backward process $Y$. A central aim of this paper is to rigorously show, under reasonable conditions and in a general setting, that this is in fact the case. We also explicitly deduce the dynamics of these adjoint backward equations. It is our intention to use this in higher order approximation schemes for coupled FBSDE, which are currently under development.

This paper is structured as follows: In section \ref{SLC} we briefly sum up the theory of decoupling fields under \emph{standard Lipschitz conditions} (SLC), which works as a brief introduction to the topic of decoupling fields in case the reader is not familiar with this approach to FBSDE.
%, which is the basis of the method of decoupling fields. 
%This works as a brief introduction to the key techniques of the method in case the reader is not familiar with this approach to FBSDE.
In section \ref{MLLC} we discuss decoupling fields under so-called \emph{modified local Lipschitz conditions} (MLLC), which is a theory derived from the one introduced in section \ref{SLC}. The MLLC theory is more appropriate for numerical schemes as the parameter functions of the system are deterministic, but do not have to be Lipschitz continuous in the control process. In section \ref{notations} we introduce some notations needed in the subsequent sections. In section \ref{firstD} we deduce the dynamics of the first spatial derivative along the forward process. In section \ref{higherD} we heuristically deduce dynamics of higher derivatives before studying them more rigorously in section \ref{mainR} (see e.g.\ Proposition \ref{roles2}). In a sense section \ref{mainR} is a generalization of the aforementioned MLLC theory to higher derivatives. The relaxation of Lipschitz continuity under MLLC comes in handy as the BSDEs satisfied by the spatial derivatives of the decoupling field are not Lipschitz continuous in general even if the initial FBSDE satisfies SLC. 

An important finding of section \ref{mainR} is that while it is possible that the first and/or the second spatial derivative of the decoupling field "explodes" at some moment in time as we move backwards away from the terminal condition, such a singularity cannot occur in the third and higher spatial derivatives if it has not already occurred in the first two and the condition $L_{\sigma,z}=0$ is satisfied (Theorem \ref{tightExp}). Thus, after having proven existence of a twice weakly differentiable and sufficiently regular decoupling field on an interval the existence and boundedness of all higher derivatives of the decoupling field follows automatically if the parameters of the problem are sufficiently smooth (Corollary \ref{highergood}, see also Remark \ref{highergoodRem}). Related statements are shown for the case $L_{\sigma,z}>0$ as well.

Finally, in section \ref{examp} we apply the theory developed in section \ref{mainR} to the FBSDE considered in \cite{Proemel2015} to recover the results on spatial differentiability of the decoupling field already proven in \cite{Proemel2015} but now reduced to a much tighter argumentation employing the theory from section \ref{mainR}. This simplification opens the door to an extension of \cite{Proemel2015} to more general diffusions, which is, however, left to future research.

\section{Decoupling fields under SLC}\label{SLC}

For a fixed finite time horizon $T>0$, we consider a complete filtered probability space $(\Omega,\mathcal{F},(\mathcal{F}_t)_{t\in[0,T]},\mathbb{P})$, where 
$\mathcal{F}_0$ consists of all null sets, $(W_t)_{t\in[0,T]}$ is a $d$-dimensional Brownian motion and $\mathcal{F}_t:=\sigma(\mathcal{F}_0,(W_s)_{s\in [0,t]})$ with $\mathcal{F}:=\mathcal{F}_T$. The dynamics of an FBSDE is given by
\begin{align*}
  X_s&=X_{0}+\int_{0}^s\mu(r,X_r,Y_r,Z_r)d r+\int_{0}^s\sigma(r,X_r,Y_r,Z_r)d W_r,\\
  Y_t&=\xi(X_T)-\int_{t}^{T}f(r,X_r,Y_r,Z_r)d r-\int_{t}^{T}Z_r dW_r,
\end{align*}
for $s,t \in [0,T]$ and $X_0 \in \mathbb{R}^n$, where $(\xi,(\mu,\sigma,f))$ are measurable functions such that 
\begin{align*}
  \xi &\colon\Omega\times\mathbb{R}^n \to \mathbb{R}^m, &
  \mu &\colon [0,T]\times\Omega\times\mathbb{R}^n\times\mathbb{R}^m\times\mathbb{R}^{m\times d}\to \mathbb{R}^n,\\
  \sigma&\colon [0,T]\times\Omega\times\mathbb{R}^n\times\mathbb{R}^m\times\mathbb{R}^{m\times d}\to \mathbb{R}^{n\times d},&
  f&\colon [0,T]\times\Omega\times\mathbb{R}^n\times\mathbb{R}^m\times\mathbb{R}^{m\times d}\to \mathbb{R}^m,
\end{align*}
for $d,n,m\in\mathbb{N}$. Throughout the whole section $\mu$, $\sigma$ and $f$ are assumed to be progressively measurable with respect to $(\mathcal{F}_t)_{t\in[0,T]}$, i.e. $\mu\mathbf{1}_{[0,t]},\sigma\mathbf{1}_{[0,t]},f\mathbf{1}_{[0,t]}$ must be $\mathcal{B}([0,T])\otimes\mathcal{F}_t\otimes\mathcal{B}(\mathbb{R}^n)\otimes\mathcal{B}(\mathbb{R}^m)\otimes\mathcal{B}(\mathbb{R}^{m\times d})$ - measurable for all $t\in[0,T]$.

A decoupling field comes with an even richer structure than just a classical solution $(X,Y,Z)$.

\begin{definition}\label{def:decoupling field}
  Let $t\in[0,T]$. A function $u\colon [t,T]\times\Omega\times\mathbb{R}^n\to\mathbb{R}^m$ with $u(T,\cdot)=\xi$ a.e. is called \emph{decoupling field} for $\fbsde (\xi,(\mu,\sigma,f))$ on $[t,T]$ if for all $t_1,t_2\in[t,T]$ with $t_1\leq t_2$ and any $\mathcal{F}_{t_1}$-measurable $X_{t_1}\colon\Omega\to\mathbb{R}^n$ there exist progressively measurable processes $(X,Y,Z)$ on $[t_1,t_2]$ such that
  \begin{align}\label{eq:decoupling}
    X_s&=X_{t_1}+\int_{t_1}^s\mu(r,X_r,Y_r,Z_r) d r+\int_{t_1}^s\sigma(r,X_r,Y_r,Z_r)d  W_r,&\nonumber\\
    Y_s&=Y_{t_2}-\int_{s}^{t_2}f(r,X_r,Y_r,Z_r) d r-\int_{s}^{t_2}Z_r d W_r,& \nonumber \\
    Y_s&=u(s,X_s),
  \end{align}
  a.s.\ for all $s\in[t_1,t_2]$. In particular, we want all integrals to be well-defined.
\end{definition}

Some remarks about this definition are in place.
\begin{itemize}
  \item The first equation in \eqref{eq:decoupling} is called the \emph{forward equation}, the second the \emph{backward equation} and the third will be referred to as the \emph{decoupling condition}.
  \item Note that, if $t_2=T$, we get $Y_T=\xi(X_T)$ a.s.\ as a consequence of the decoupling condition together with $u(T,\cdot)=\xi$. At the same time $Y_T=\xi(X_T)$ together with the decoupling condition implies $u(T,\cdot)=\xi$ a.e.
  \item If $t_2=T$ we can say that a triplet $(X,Y,Z)$ solves the FBSDE, meaning that it satisfies the forward and the backward equation, together with $Y_T=\xi(X_T)$. This relationship $Y_T=\xi(X_T)$ is referred to as the \emph{terminal condition}. 
\end{itemize}

In contrast to classical solutions of FBSDEs, decoupling fields on adjacent intervals can be pasted together (see e.g.\ Lemma 2.1.2 of \cite{Fromm2015}).

%\begin{lemma}[\cite{Fromm2015}, Lemma 2.1.2]\label{glue}
%  Let $u$ be a decoupling field for $\fbsde (\xi,(\mu,\sigma,f))$ on $[t,T]$ and $\tilde{u}$ be a decoupling field for $\fbsde (u(t,\cdot),(\mu,\sigma,f))$ on $[s,t]$, for $0\leq s<t<T$. Then, the map $\hat{u}$ given by $\hat{u}:=\tilde{u}\mathbf{1}_{[s,t]}+u\mathbf{1}_{(t,T]}$ is a decoupling field for $\fbsde (\xi,(\mu,\sigma,f))$ on $[s,T]$.
%\end{lemma}

We want to remark that, if $u$ is a decoupling field and $\tilde{u}$ is a modification of $u$, i.e. for each $s\in[t,T]$ the functions $u(s,\omega,\cdot)$ and $\tilde{u}(s,\omega,\cdot)$ coincide for almost all $\omega\in\Omega$, then $\tilde{u}$ is also a decoupling field to the same problem. 
Hence, $u$ could also be referred to as a class of modifications and a progressively measurable and in some sense 
right-continuous representative exists if the decoupling field is Lipschitz continuous in $x$ (Lemma 2.1.3 in \cite{Fromm2015}).

For the following we need to fix further notation.

Let $I\subseteq [0,T]$ be an interval and $u: I\times\Omega\times\mathbb{R}^n\rightarrow \mathbb{R}^m$ a map such that $u(s,\cdot)$ is measurable for every $s\in I$. We define
\begin{equation*}
  L_{u,x}:=\sup_{s\in I}\inf\{L\geq 0\,|\,\textrm{for a.a. }\omega\in\Omega: |u(s,\omega,x)-u(s,\omega,x')|\leq L|x-x'|\textrm{ for all }x,x'\in\mathbb{R}^n\},
\end{equation*}
where $\inf \emptyset:=\infty$. We also set $ L_{u,x}:=\infty$ if $u(s,\cdot)$ is not measurable for every $s\in I$. One can show that $L_{u,x}<\infty$ is equivalent to $u$ having a modification which is truly Lipschitz continuous in $x\in\mathbb{R}^n$.

We denote by $L_{\sigma,z}$ the Lipschitz constant of $\sigma$ w.r.t.\ the dependence on the last component $z$ and w.r.t.\ the Frobenius norms on $\mathbb{R}^{m\times d}$ and $\mathbb{R}^{n\times d}$. We set $L_{\sigma,z}=\infty$ if $\sigma$ is not Lipschitz continuous in $z$. 

By $L_{\sigma,z}^{-1}=\frac{1}{L_{\sigma,z}}$ we mean $\frac{1}{L_{\sigma,z}}$ if $L_{\sigma,z}>0$ and $\infty$ otherwise.

For an integrable real valued random variable $F$ the expression $\mathbb{E}_t[F]$ refers to $\mathbb{E}[F|\mathcal{F}_t]$, while $\mathbb{E}_{t,\infty}[F]$ refers to $\esssup\,\mathbb{E}[F|\mathcal{F}_t]$, which might be $\infty$, but is always well defined as the infimum of all constants $c\in[-\infty,\infty]$ such that $\mathbb{E}[F|\mathcal{F}_t]\leq c$ a.s. Additionally, we write $\|F\|_\infty$ for the essential supremum of $|F|$.

Finally for a matrix $A\in\mathbb{R}^{N\times n}$ and a vector $v\in S^{n-1}$ we define $|A|_v:=|Av|$ as the norm of $A$ in the direction $v$, where $S^{n-1}$ is the $(n-1)$ - dimensional sphere. 

In practice it is important to have explicit knowledge about the regularity of $(X,Y,Z)$. For instance, it is important to know in which spaces the processes live, and how they react to changes in the initial value. 

\begin{definition}\label{regulSLC}
  Let $u\colon [t,T]\times\Omega\times\mathbb{R}^n\to\mathbb{R}^m$ be a decoupling field to $\fbsde(\xi,(\mu,\sigma,f))$.
  \begin{enumerate}
   \item We say $u$ to be \emph{weakly regular} if $L_{u,x}<L_{\sigma,z}^{-1}$ and $\sup_{s\in[t,T]}\|u(s,\cdot,0)\|_{\infty}<\infty$.
   \item A weakly regular decoupling field $u$ is called \emph{strongly regular} if for all fixed $t_1,t_2\in[t,T]$, $t_1\leq t_2,$ the processes $(X,Y,Z)$ arising in \eqref{eq:decoupling}  are a.e.\ unique and satisfy
   \begin{equation}\label{strongregul1}
      \sup_{s\in [t_1,t_2]}\mathbb{E}_{t_1,\infty}[|X_s|^2]+\sup_{s\in [t_1,t_2]}\mathbb{E}_{t_1,\infty}[|Y_s|^2]
      +\mathbb{E}_{t_1,\infty}\left[\int_{t_1}^{t_2}|Z_s|^2 d s\right]<\infty,
   \end{equation}
   for each constant initial value $X_{t_1}=x\in\mathbb{R}^n$. In addition they are required to be measurable as functions of $(x,s,\omega)$ and even weakly differentiable w.r.t.\ $x\in\mathbb{R}^n$ such that for every $s\in[t_1,t_2]$ the mappings $X_s$ and $Y_s$ are measurable functions of $(x,\omega)$ and even weakly differentiable w.r.t.\ $x$ such that
   \begin{align}\label{strongregul2}
     &\esssup_{x\in\mathbb{R}^n}\sup_{v\in S^{n-1}}\sup_{s\in [t_1,t_2]}\mathbb{E}_{t_1,\infty}\left[\left|\frac{\dx}{\dx x}X_s\right|^2_v\right]<\infty, \nonumber\allowdisplaybreaks\\
     &\esssup_{x\in\mathbb{R}^n}\sup_{v\in S^{n-1}}\sup_{s\in [t_1,t_2]}\mathbb{E}_{t_1,\infty}\left[\left|\frac{\dx}{\dx x}Y_s\right|^2_v\right]<\infty, \nonumber\\
     &\esssup_{x\in\mathbb{R}^n}\sup_{v\in S^{n-1}}\mathbb{E}_{t_1,\infty}\left[\int_{t_1}^{t_2}\left|\frac{\dx}{\dx x}Z_s\right|^2_v\dx s\right]<\infty.
   \end{align}
  % \item We say that a decoupling field on $[t,T]$ is \emph{strongly regular} on a subinterval $[t_1,t_2]\subseteq[t,T]$ if $u$ restricted to $[t_1,t_2]$ is a strongly regular decoupling field for $\fbsde(u(t_2,\cdot),(\mu,\sigma,f))$.
   \end{enumerate}
\end{definition}

Under suitable conditions a rich existence, uniqueness and regularity theory for decoupling fields can be developed. The basis of the theory is Theorem \ref{locexist} below, which is
proven in Chapter 2 of \cite{Fromm2015}.
%We will summarize the main results, which are proven in Chapter 2 of \cite{Fromm2015}:
\smallskip\\
{\bf Assumption (SLC):} $(\xi,(\mu,\sigma,f))$ satisfies \emph{standard Lipschitz conditions} \textup{(SLC)} if
\begin{enumerate}
  \item $(\mu,\sigma,f)$ are Lipschitz continuous in $(x,y,z)$ with Lipschitz constant $L$,
  \item $\left\|\left(|\mu|+|f|+|\sigma|\right)(\cdot,\cdot,0,0,0)\right\|_{\infty}<\infty$,
  \item $\xi\colon \Omega\times\mathbb{R}^n\to \mathbb{R}^m$ is measurable such that $\|\xi(\cdot,0)\|_{\infty}<\infty$ and $L_{\xi,x}<L_{\sigma,z}^{-1}$.
\end{enumerate}

\begin{thm}[\cite{Fromm2015}, Theorem 2.2.1]\label{locexist}
  Suppose $(\xi,(\mu,\sigma,f))$ satisfies \textup{(SLC)}. Then there exists a time $t\in[0,T)$ such that $\fbsde (\xi,(\mu,\sigma,f))$ has a unique (up to modification) decoupling field $u$ on $[t,T]$ with $L_{u,x}<L_{\sigma,z}^{-1}$ and $\sup_{s\in [t,T]}\|u(s,\cdot,0)\|_{\infty}<\infty$.
\end{thm}

A brief discussion of existence and uniqueness of classical solutions on sufficiently small intervals can be found in Remark 2.2.4 in \cite{Fromm2015}. 

This local theory for decoupling fields can be systematically extended to global results based on fairly simple ``small interval induction'' arguments (Lemma 2.5.1 and 2.5.2 in \cite{Fromm2015}).
%\begin{thm}[\cite{Fromm2015}, Corollary 2.5.3, 2.5.4 and 2.5.5]\label{uniq}
%  Suppose that $(\xi,(\mu,\sigma,f))$ satisfies \textup{(SLC)}.
%  \begin{enumerate}
%    \item\textup{Global uniqueness:} If there are two weakly regular decoupling fields $u^{(1)},u^{(2)}$ to the corresponding problem on some interval $[t,T]$, then we have $u^{(1)}=u^{(2)}$ up to modifications.
%    \item\textup{Global regularity:} If there exists a weakly regular decoupling field $u$ to this problem on some interval $[t,T]$, then $u$ is strongly regular.
%    \item If there exists a weakly regular decoupling field $u$ of the corresponding FBSDE on some interval $[t,T]$, then for any initial condition $X_t=x\in\mathbb{R}^n$ there is a unique solution $(X,Y,Z)$ of the FBSDE on $[t,T]$ satisfying
%    \begin{equation*}
%      \sup_{s\in[t,T]}\mathbb{E}[|X_s|^2]+\sup_{s\in[t,T]}\mathbb{E}[|Y_s|^2]+\mathbb{E}\left[\int_t^T|Z_s|^2 d s\right]<\infty.
%    \end{equation*}
%  \end{enumerate}
%\end{thm}
In order to have a notion of global existence we need the following definition:

\begin{definition} We define the maximal interval $I_{\max}\subseteq[0,T]$ of the problem given by  $(\xi,(\mu,\sigma,f))$ as the union of all
intervals $[t,T]\subseteq[0,T]$, such that there exists a weakly regular decoupling field $u$ on $[t,T]$.
\end{definition}

Note that the maximal interval might be open to the left. Also, let us remark that we define a decoupling field on such an interval as a mapping which is a decoupling field
on every compact subinterval containing $T$. Similarly we can define weakly and strongly regular decoupling fields as mappings which restricted to an arbitrary
compact subinterval containing $T$ are weakly (or strongly) regular decoupling fields in the sense of the definitions given above.

Finally, we have global existence and uniqueness on the maximal interval:

\begin{thm}[Global existence in weak form, \cite{Fromm2015}, Theorem 5.1.11 and Lemma 5.1.12]\label{globalexist}
 Let $(\xi,(\mu,\sigma,f))$ satisfy SLC. Then there exists a
unique weakly regular decoupling field $u$ on $I_{\max}$. This $u$ is even strongly regular.
Furthermore, either $I_{\max}=[0,T]$ or $I_{\max}=(t_{\min},T]$, where $0 \leq t_{\min} < T$. In the latter case we have $\lim_{t\downarrow t_{\min}} L_{u(t,\cdot),x}=L_{\sigma,z}^{-1}$.
%\ben\label{explosion}\lim_{t\downarrow t_{\min}} L_{u(t,\cdot),x}=L_{\sigma,z}^{-1}.\een
\end{thm}
%Note that in particular cases the last statement allows to show ``strong global existence'', i.e. $I_{\max}=[0,T]$, via contradiction and, thereby, is the basis of the so-called method of decoupling fields. Let us describe the different steps:
%\begin{enumerate}
%\item Assume indirectly that $I_{\mathrm{max}}=[0,T]$ does not hold, which implies $I_{\mathrm{max}}=(t_{\mathrm{min}},T]$. Choose arbitrary $t\in (t_{\mathrm{min}},T]$, $x\in\mathbb{R}^n$ and consider the corresponding FBSDE.
%\item Differentiate the FBSDE w.r.t.\ $x$. This is possible because of strong regularity of $u$ (Theorem \ref{globalexist}). We obtain joint dynamics of $\frac{\dx }{\dx x}X,\frac{\dx }{\dx x}Y,\frac{\dx }{\dx x}Z$.
%\item Using It\^o's formula deduce the dynamics of $\frac{\dx }{\dx x}Y_s(\frac{\dx }{\dx x}X_s)^{-1}$. This process can be expected to coincide with $u_x(s,X_s)$, as a consequence of the chain rule applied to the decoupling condition $Y_s=u(s,X_s)$.
%\item Using the dynamics of $u_x(s,X_s)$ show that its modulus can be bounded away from $L_{\sigma,z}^{-1}$ independently of $t,x,s,\omega$. This contradicts \eqref{explosion} and, therefore, $I_{\mathrm{max}}=[0,T]$ must hold.
%\end{enumerate}

\section{Decoupling fields under MLLC}\label{MLLC}

In this section we briefly summarize the key results of the abstract theory of Markovian decoupling fields, we rely on later in the paper. The presented theory is derived from the SLC theory of Chapter 2 of \cite{Fromm2015} and is proven in \cite{Proemel2015}.

%We consider families $(\mu,\sigma,f)$ of measurable functions, more precisely
%$$ \mu: [0,T]\times\Omega\times\mathbb{R}^n\times\mathbb{R}^m\times\mathbb{R}^{m\times d}\longrightarrow \mathbb{R}^n, $$
%$$ \sigma: [0,T]\times\Omega\times\mathbb{R}^n\times\mathbb{R}^m\times\mathbb{R}^{m\times d}\longrightarrow \mathbb{R}^{n\times d}, $$
%$$ f: [0,T]\times\Omega\times\mathbb{R}^n\times\mathbb{R}^m\times\mathbb{R}^{m\times d}\longrightarrow \mathbb{R}^m, $$
%where $n,m,d\in\mathbb{N}$ and $T>0$, $(\Omega,\mathcal{F},\mathbb{P},(\mathcal{F}_t)_{t\in[0,T]})$ is a complete filtered probability space,
%such that $\mathcal{F}_0$ consists of all null sets, $\mathcal{F}=\mathcal{F}_T$ and
%$\mathcal{F}_t=\sigma(\mathcal{F}_0,(W_s)_{s\in [0,t]})$ for all $t\in [0,T]$ holds,
 %where $(W_t)_{t\in[0,T]}$ is a $d$-dimensional Brownian motion.

Let $(\Omega,\mathcal{F},\mathbb{P},(\mathcal{F}_t)_{t\in[0,T]})$ be as in the previous section. Again, we consider progressively measurable mappings $\mu$, $\sigma$, $f$ and a measurable $\xi$ with the same domains and target spaces as in section \ref{SLC}.
%Again, we consider progressively measurable mappings $\mu$, $\sigma$ and $f$, where
%$$ \mu: [0,T]\times\Omega\times\mathbb{R}^n\times\mathbb{R}^m\times\mathbb{R}^{m\times d}\longrightarrow \mathbb{R}^n, $$
%$$ \sigma: [0,T]\times\Omega\times\mathbb{R}^n\times\mathbb{R}^m\times\mathbb{R}^{m\times d}\longrightarrow \mathbb{R}^{n\times d}, $$
%$$ f: [0,T]\times\Omega\times\mathbb{R}^n\times\mathbb{R}^m\times\mathbb{R}^{m\times d}\longrightarrow \mathbb{R}^m. $$
A problem given by $\xi,\mu,\sigma,f$ is said to be \emph{Markovian}, if these four functions are deterministic, i.e. depend on $t,x,y,z$ only.
In the Markovian case we can somewhat relax the Lipschitz continuity assumptions (SLC) made in the previous section and still obtain local existence together with uniqueness. What makes the Markovian case so special is the property
$$"Z_s=u_x(s,X_s)\cdot\sigma(s,X_s,Y_s,Z_s)"$$
which comes from the fact that $u$ becomes deterministic as well. This property allows us to bound $Z$ by a constant if we assume that $\sigma$ is bounded.

This potential boundedness of $Z$ in the Markovian case motivates the following definition, which allows to develop a theory for non-Lipschitz problems:

\begin{defi}
Let $\xi:\Omega\times\mathbb{R}^n\rightarrow\mathbb{R}^m$ be measurable and let $t\in[0,T]$.\\
We call a function $u:[t,T]\times\Omega\times\mathbb{R}^n\rightarrow\mathbb{R}^m$ with $u(T,\omega,\cdot)=\xi(\omega,\cdot)$ for a.a. $\omega\in\Omega$ a \emph{Markovian decoupling field} for $\fbsde (\xi,(\mu,\sigma,f))$ on $[t,T]$ if for all $t_1,t_2\in[t,T]$ with $t_1\leq t_2$ and any $\mathcal{F}_{t_1}$ - measurable $X_{t_1}:\Omega\rightarrow\mathbb{R}^n$ there exist progressive processes $X,Y,Z$ on $[t_1,t_2]$ such that the three equations \eqref{eq:decoupling} are satisfied a.s.\
%\begin{itemize}
%\item $X_s=X_{t_1}+\int_{t_1}^s\mu(r,X_r,Y_r,Z_r)\dx r+\int_{t_1}^s\sigma(r,X_r,Y_r,Z_r)\dx W_r$ a.s.,
%\item $Y_s=Y_{t_2}-\int_{s}^{t_2}f(r,X_r,Y_r,Z_r)\dx r-\int_{s}^{t_2}Z_r\dx W_r$ a.s.,
%\item $Y_s=u(s,X_s)$ a.s.
%\end{itemize}
for all $s\in[t_1,t_2]$ \underline{and such that $\|Z\|_\infty<\infty$ holds}.\\ In particular, we want all integrals to be well-defined and $X,Y,Z$ to have values in $\mathbb{R}^n$, $\mathbb{R}^m$ and $\mathbb{R}^{m\times d}$ respectively. \\
Furthermore, we call a function $u:(t,T]\times\mathbb{R}^n\rightarrow\mathbb{R}^m$ a Markovian decoupling field for $\fbsde(\xi,(\mu,\sigma,f))$ on $(t,T]$ if $u$ restricted to $[t',T]$ is a Markovian decoupling field for all $t'\in(t,T]$.
\end{defi}

A Markovian decoupling field is always a decoupling field in the standard sense as well. The only difference between the two notions is that we are only interested in $X,Y,Z$, where $Z$ is a.e.\ bounded.
Regularity for Markovian decoupling fields is defined very similarly to standard regularity: 

%We write $\mathbb{E}_{t,\infty}[X]$ for $\esssup\,\mathbb{E}[X|\mathcal{F}_t]$ in the following definition:

\begin{defi}\label{regulMLLC}
Let $u:[t,T]\times\Omega\times\mathbb{R}^n\rightarrow\mathbb{R}^m$ be a Markovian decoupling field to $\fbsde(\xi,(\mu,\sigma,f))$. We call $u$ \emph{weakly regular}, if
$L_{u,x}<L_{\sigma,z}^{-1}$ and $\sup_{s\in[t,T]}\|u(s,\cdot,0)\|_{\infty}<\infty$.

Furthermore, we call a weakly regular $u$ \emph{strongly regular} if for all fixed $t_1,t_2\in[t,T]$, $t_1\leq t_2,$ the processes $X,Y,Z$ arising in the defining property of a Markovian decoupling field
are a.e. unique for each \emph{constant} initial value $X_{t_1}=x\in\mathbb{R}^n$ and satisfy \eqref{strongregul1}. \\
%\begin{equation}\label{STRongregul1}
%\sup_{s\in [t_1,t_2]}\mathbb{E}_{t_1,\infty}[|X_s|^2]+\sup_{s\in [t_1,t_2]}\mathbb{E}_{t_1,\infty}[|Y_s|^2]
%+\mathbb{E}_{t_1,\infty}\left[\int_{t_1}^{t_2}|Z_s|^2\dx s\right]<\infty\quad\forall x\in\mathbb{R}^n.
%\end{equation}
 In addition they must be measurable as functions of $(x,s,\omega)$ and
even weakly differentiable w.r.t.\ $x\in\mathbb{R}^n$ such that for every $s\in[t_1,t_2]$ the mappings $X_s$ and $Y_s$ are measurable functions of $(x,\omega)$ and even weakly differentiable w.r.t.\ $x$ such that \eqref{strongregul2} is satisfied.

%\begin{multline}\label{STRongregul2}
%\esssup_{x\in\mathbb{R}^n}\sup_{v\in S^{n-1}}\sup_{s\in [t_1,t_2]}\mathbb{E}_{t_1,\infty}\left[\left|\frac{\dx}{\dx x}X_s\right|^2_v\right]<\infty, \\
%\esssup_{x\in\mathbb{R}^n}\sup_{v\in S^{n-1}}\sup_{s\in [t_1,t_2]}\mathbb{E}_{t_1,\infty}\left[\left|\frac{\dx}{\dx x}Y_s\right|^2_v\right]<\infty, \\
%\esssup_{x\in\mathbb{R}^n}\sup_{v\in S^{n-1}}\mathbb{E}_{t_1,\infty}\left[\int_{t_1}^{t_2}\left|\frac{\dx}{\dx x}Z_s\right|^2_v\dx s\right]<\infty,
%\end{multline}
%We say that a Markovian decoupling field $u$ on $[t,T]$ is \emph{strongly regular} on a subinterval $[t_1,t_2]\subseteq[t,T]$ if $u$ restricted to $[t_1,t_2]$ is a strongly regular Markovian decoupling field for $\fbsde(u(t_2,\cdot),(\mu,\sigma,f))$.
We can define weakly and strongly regular Markovian decoupling fields on a half-open interval $(t,T]$ as mappings which restricted to an arbitrary
compact subinterval containing $T$ are weakly (or strongly) regular Markovian decoupling fields in the sense of the definitions given above.
%Furthermore, we say that a Markovian decoupling field $u:(t,T]\times\mathbb{R}^n\rightarrow\mathbb{R}^m$
%\begin{itemize}
%\item is weakly regular if $u$ restricted to $[t',T]$ is weakly regular for all $t'\in(t,T]$,
%\item is strongly regular if $u$ restricted to $[t',T]$ is strongly regular for all $t'\in(t,T]$.
%\end{itemize}
\end{defi}
For the following class of problems an existence and uniqueness theory is developed:
\begin{defi}
We say that $\xi,\mu,\sigma,f$ satisfy \emph{modified local Lipschitz conditions (MLLC)} if
\begin{itemize}
\item $\mu,\sigma,f$ are
\begin{itemize}
\item deterministic,
\item Lipschitz continuous in $x,y,z$ on sets of the form $[0,T]\times\mathbb{R}^n\times\mathbb{R}^{m} \times B$, where $B\subset \mathbb{R}^{m\times d}$ is an arbitrary bounded set
\item and such that $\|\mu(\cdot,0,0,0)\|_\infty,\|f(\cdot,0,0,0)\|_{\infty},\|\sigma(\cdot,\cdot,\cdot,0)\|_{\infty},L_{\sigma,z}<\infty$,
\end{itemize}
\item $\xi: \mathbb{R}^n\rightarrow \mathbb{R}^m$ satisfies $L_{\xi,x}<L_{\sigma,z}^{-1}$,
\end{itemize}
where $L_{\sigma,z}$ denotes the Lipschitz constant of $\sigma$ w.r.t.\ the dependence on the last component $z$ (and w.r.t.\ the Frobenius norms on $\mathbb{R}^{m\times d}$ and $\mathbb{R}^{n\times d}$).
%By $L_{\sigma,z}^{-1}=\frac{1}{L_{\sigma,z}}$ we mean $\frac{1}{L_{\sigma,z}}$ if $L_{\sigma,z}>0$ and $\infty$ otherwise.
\end{defi}

The following natural concept introduces a type of Markovian decoupling field for non-Lipschitz problems (non-Lipschitz in $z$), to which nevertheless standard Lipschitz results can be applied.

\begin{defi}
Let $u$ be a Markovian decoupling field for $\fbsde(\xi,(\mu,\sigma,f))$. We call $u$ \emph{controlled in $z$} if there exists a constant $C>0$ such that for all $t_1,t_2\in[t,T]$, $t_1\leq t_2$, and all initial values $X_{t_1}$, the corresponding processes $X,Y,Z$ from the definition of a Markovian decoupling field satisfy $|Z_s(\omega)|\leq C$,
for almost all $(s,\omega)\in[t,T]\times\Omega$. If for a fixed triple $(t_1,t_2,X_{t_1})$ there are different choices for $X,Y,Z$, then all of them are supposed to satisfy the above control.

We say that a Markovian decoupling field $u$ on $[t,T]$ is \emph{controlled in $z$} on a subinterval $[t_1,t_2]\subseteq[t,T]$ if $u$ restricted to $[t_1,t_2]$ is a Markovian decoupling field for $\fbsde(u(t_2,\cdot),(\mu,\sigma,f))$ that is controlled in $z$.

Furthermore, we call a Markovian decoupling field on an interval $(s,T]$ \emph{controlled in $z$} if it is controlled in $z$ on every compact subinterval $[t,T]\subseteq (s,T]$ (with $C$ possibly depending on $t$).
\end{defi}

The following important result allows us to connect the MLLC - case to SLC.

\begin{thm}[\cite{Proemel2015}, Theorem 3.16]\label{CONtrollM}
Let $\xi,\mu,\sigma,f$ satisfy (MLLC) and assume that there exists a weakly regular Markovian decoupling field $u$ to this problem on some interval $[t,T]$. Then $u$ is controlled in $z$.
\end{thm}

%Note at this point that such a $u$ will be a standard decoupling field to an SLC problem if we cutoff $\mu,\sigma,f$ appropriately. We can, thereby, extend the whole SLC theory to MLLC problems:
As applications of analogous results for SLC problems from Chapter 2 of \cite{Fromm2015} one shows

\begin{thm}[\cite{Proemel2015}, Theorem 3.17 and Theorem 3.18]\label{UNIqMREGulM}
Let $\xi,\mu,\sigma,f$ satisfy (MLLC).
\begin{enumerate}
\item Assume that there are two weakly regular Markovian decoupling fields $u^{(1)},u^{(2)}$ to this problem on some interval $[t,T]$.
Then $u^{(1)}=u^{(2)}$ (up to modifications).
\item Assume that there exists a weakly regular Markovian decoupling field $u$ to this problem on some interval $[t,T]$. Then $u$ is strongly regular.
\item Assume that there exists a weakly regular Markovian decoupling field $u$ on some interval $[t,T]$. Then for any initial condition $X_t=x\in\mathbb{R}^n$ there is a unique solution $(X,Y,Z)$ of the FBSDE on $[t,T]$ such that
$$\sup_{s\in[t,T]}\mathbb{E}[|X_s|^2]+\sup_{s\in[t,T]}\mathbb{E}[|Y_s|^2]+\|Z\|_\infty<\infty.$$
\end{enumerate}
\end{thm}

\begin{defi}
Let $J_{\mathrm{max}}\subseteq[0,T]$ for $\fbsde(\xi,(\mu,\sigma,f))$ be the union of all intervals $[t,T]\subseteq[0,T]$ such that there exists a weakly regular Markovian decoupling field $u$ on $[t,T]$.
\end{defi}

\begin{thm}[Global existence in weak form, \cite{Proemel2015}, Theorem 3.21]\label{GLObalexistM}
Let $\xi,\mu,\sigma,f$ satisfy (MLLC). Then there exists a unique weakly regular Markovian decoupling field $u$ on $J_{\mathrm{max}}$. This $u$ is also controlled in $z$, strongly regular, deterministic and continuous. \\
Furthermore, either $J_{\mathrm{max}}=[0,T]$ or $J_{\mathrm{max}}=(s_{\mathrm{min}},T]$, where $0\leq s_{\mathrm{min}}<T$.
\end{thm}

The following result basically states that for a singularity $s_{\mathrm{min}}$ to occur $u_x$ has to "explode" at $s_{\mathrm{min}}$. It is the key to showing well-posedness for particular problems via contradiction.

\begin{thm}[\cite{Proemel2015}, Lemma 3.22]\label{EXPlosionM}
Let $\xi,\mu,\sigma,f$ satisfy (MLLC). If $J_{\mathrm{max}}=(s_{\mathrm{min}},T]$, then
$$\lim_{t\downarrow s_{\mathrm{min}}}L_{u(t,\cdot),x}=L_{\sigma,z}^{-1},$$
where $u$ is the unique weakly regular Markovian decoupling field from Theorem \ref{GLObalexistM}.
\end{thm}

\section{Some important notations}\label{notations}

Let us introduce some notions and notations which are used in the subsequent sections.

Firstly, we work with generalized matrices: Let $k\in\mathbb{N}$ and $n_1,\ldots,n_k\in\mathbb{N}$ be natural numbers. For every $m\in\mathbb{N}$ we define $[m]:=\{1,\ldots,m\}$ which is a set of cardinality $m$. Now we define the space $\mathbb{R}^{n_1\times\ldots\times n_k}$ as the linear space of all mappings $A:[n_1]\times\ldots\times[n_k]\rightarrow\mathbb{R}$. Note that in case $k=2$ the mapping $A$ is a standard matrix. For $k=1$ we have a standard vector.

We can multiply generalized matrices: Let $A\in\mathbb{R}^{n_1\times\ldots\times n_k}$ and $B\in\mathbb{R}^{m_1\times\ldots\times m_l}$, where $k,l\in\mathbb{N}$, and assume that $m_1=n_k$. Then the product $A\cdot B$ is defined as the mapping $C\in \mathbb{R}^{n_1\times\ldots\times n_{k-1} \times m_2\times\ldots\times m_l}$ which satisfies 
$$ C(x_1,\ldots, x_{k-1},y_{2},\ldots,y_{l}):=\sum_{z=1}^{m_1}A(x_1,\ldots, x_{k-1}, z)B(z,y_2,\ldots,y_l), $$
where $x_i\in[n_i]$ for all $i=1,\ldots,k-1$ and where $y_i\in[m_i]$ for all $i=2,\ldots,l$.

This definition is consistent with standard matrix multiplication in case $k=l=2$. We may at times write $AB$ instead of $A\cdot B$ for simplicity.

A particularity of our analysis in this work is that we also consider products $A\cdot B$, $A\in\mathbb{R}^{n_1\times\ldots\times n_k}$ , $B\in\mathbb{R}^{m_1\times\ldots\times m_l}$, where $(n_{k-1},n_k)=(m_1,m_2)=(m,d)$, where $m,d$ are two fixed natural numbers which in the subsequent sections will have the same role is in sections \ref{SLC} and \ref{MLLC}. This product is defined as the mapping $C\in \mathbb{R}^{n_1\times\ldots\times n_{k-2} \times m_3\times\ldots\times m_l}$ which satisfies 
$$ C(x_1,\ldots, x_{k-2},y_{3},\ldots,y_{l}):=\sum_{(z_1,z_2)\in[m]\times[d]}A(x_1,\ldots, x_{k-2}, z_1,z_2)B(z_1,z_2,y_3,\ldots,y_l), $$
where $x_i\in[n_i]$ for all $i=1,\ldots,k-2$ and where $y_i\in[m_i]$ for all $i=3,\ldots,l$. 
We still write $A\cdot B$ for this product but indicate the application of it by writing $A\in\mathbb{R}^{n_1\times\ldots\times n_{k-2}\times (m\times d)}$ instead of simply $A\in\mathbb{R}^{n_1\times\ldots\times n_{k-2}\times m\times d}$ or by writing $B\in\mathbb{R}^{(m\times d) \times m_3\times\ldots\times m_l}$ instead of $B\in\mathbb{R}^{m\times d \times m_3\times\ldots\times m_l}$, or by setting the brackets in both. 

Note that if $A\in\mathbb{R}^{(m\times d) \times (m\times d)}$ and $B\in\mathbb{R}^{m\times d}$, then $B\mapsto A\cdot B$ describes a linear mapping on and to the linear space $\mathbb{R}^{m\times d}$. This linear mapping is invertible if and only if there exists an $A^{-1}\in\mathbb{R}^{(m\times d) \times (m\times d)}$ such that $C:=A\cdot A^{-1}\in\mathbb{R}^{(m\times d) \times (m\times d)}$ has the property that the associated linear mapping on $\mathbb{R}^{m\times d}$ is the identity mapping.

Now consider some $A\in\mathbb{R}^{n_1\times\ldots\times n_k}$ and assume that $n_i=n\in\mathbb{N}$ for all $i=l,\ldots,k$, where $l\in [k]$. Then we may at times write $A\in\mathbb{R}^{n_1\times\ldots\times n_l\times_{k-l+1}\, n}$ instead. In other words the subscript after the last "$\times$" indicates how often "$\times n$" is applied at the end. We may also write "$\times_{i}\, n$" with $i=0$ which means that there is no "$\times n$" at the end at all.

Now let $A\in\mathbb{R}^{n_1\times\ldots\times n_l \times_{i} n}$ for some $i\geq 1$. Then we can canonically identify $A$ with a sequence $A_1,\ldots,A_n$ of  generalized matrices from $\mathbb{R}^{n_1\times\ldots\times n_l \times_{i-1} n}$, i.e.\ real-valued mappings on $[n_1]\times\ldots\times [n_l]\times [n]^{i-1}$, by writing $[n]^i=[n]^{i-1}\times[n]$ and running through the trailing $[n]=\{1,\ldots,n\}$. This identification is useful for various reasons: For instance, we can rewrite a product $C=A\cdot B\in \mathbb{R}^{n_1\times\ldots\times n_l \times_{i} n}$, for arbitrary $B=(B_{jk})\in \mathbb{R}^{n\times n}$ by setting $C_k=\sum_{j=1}^n A_j B_{jk}$, such that $C_1,\ldots,C_n$ form $C$ in the same sense that $A_1,\ldots,A_n$ form $A$. Similarly, we can express products $A\cdot B\in \mathbb{R}^{n_1\times\ldots\times n_l \times_{i-1} n}$, where $B\in\mathbb{R}^n$. Also, for the case $l=2$ and $i\geq 1$, this decomposition allows to define the Frobenius norm $|A|_2$ recursively be setting $|A|_2:=\sqrt{\sum_{j=1}^n |A_j|^2_2}$ using the fact that for normal matrices the Frobenius norm is already defined.

Secondly, in some proofs we use the following notation: Assume we have a filtered probability space generated by a Brownian motion as in section \ref{SLC}. We denote by $\tdiv P$ the density in the finite variation part of an It\^o process $P$ and by $\bmdiv P$ the density in the martingale part of an It\^o process. We denote by $\bmdiv^i P$ the $i$-th component of $\bmdiv P$, i.e. the component which is multiplied by $\dx W^{i}_t$, $i=1,\ldots,d$, in the stochastic integral. This notation helps to shorten some calculations as the product rule for two It\^o processes $P$ and $Q$ now assumes the form
$$ \bmdiv^i (P\cdot Q) = (\bmdiv^i P)\cdot Q+ P\cdot \bmdiv^i Q, $$
$$ \tdiv (P\cdot Q) = (\tdiv P)\cdot Q+ P\cdot \tdiv Q+\sum_{i=1}^d (\bmdiv^i P)\cdot (\bmdiv^i Q). $$
Here $P,Q$ may be generalized matrices such that their product is well-defined. Note that $\tdiv P,\bmdiv^i P$ assume values in the same space as $P$.

Thirdly, we sometimes use the term "inner cutoff in a variable" to describe the following manipulation of a function $f$, which depends on, let's say, two variables $x,y$, where $y$ is from an Euclidean space $E$: Let $\chi$ be a Lipschitz continuous and bounded function on and to this Euclidean space for which there is a compact zero-centred ball $B\subset E$ such that $\chi(y)\in B$ is the projection of a given $y\in E$ to the convex set $B$. We can now define the manipulated function $\tilde{f}$ via $\tilde{f}(x,y):=f(x,\chi(y))$. In a sense the dependence on $y$ is "cut off" at some level which depends on the cutoff function $\chi$. We call the cutoff \emph{passive} for a $y\in E$ if $\chi(y)=y$.

\section{Dynamics of the first derivative}\label{firstD}

Assume that for given $\xi,\mu,\sigma,f$ satisfying (SLC) or (MLLC) we have a weakly regular decoupling field $u$ on an interval $[t_0,T]$. As $u$ is Lipschitz continuous in the spatial component there exists the spatial derivative $u_x$ defined as the classical derivative or as $0$ depending on whether the classical derivative exists or not (Lipschitz continuous functions are differentiable almost everywhere). Our objective is to study the dynamics of the $\mathbb{R}^{m\times n}$ - valued process $V_s:=u_x(s,X_s)$, where $X$ is the forward process for some initial condition $X_{t_1}=x\in\mathbb{R}^n$, where $t_1\in[t_0,T)$ and $s\in[t_1,T]$. Note that $u_x(T,\cdot)=\xi':\Omega\times\mathbb{R}^n\rightarrow\mathbb{R}^{m\times n}$ is known. 
Note also that $\|u_x\|_\infty = L_{u,x}$, where we take the essential supremum of the operator norm of $u_x$ (by which we mean the operator norm w.r.t.\ the Euclidean norms on $\mathbb{R}^{n}$ and $\mathbb{R}^{m}$). This implies $\|V\|_\infty\leq L_{u,x}$, again w.r.t.\ the operator norm.
%For the following we work with generalized matrices:
%\begin{itemize}
% \item For any $a,b,c\in\mathbb{N}$ we define $\IR^{a\times b \times c}$ as the set of all objects $(p_{ijk})_{i\in [a],j\in [b], k\in[c]}$, where $[m]:=\{1,\ldots,m\}$ for all $m\in\mathbb{N}$ and where $p_{ijk}\in\mathbb{R}$. If $A \in \IR^{a\times b \times c}$ and $B=(q_{ij})_{i\in [c],j\in [d]} \in \IR^{c \times d}$ is a normal matrix, we can define the product $A\cdot B\in \IR^{a\times b \times d}$ as the object $(r_{ijk})_{i\in [a],j\in [b], k\in[d]}$, where $r_{ijk}:=\sum_{l=1}^c p_{ijl}q_{lk}$. Similarly, we can define products $A\cdot B\in \IR^{d\times b \times c}$ for $A\in \IR^{d \times a}$ and $B\in \IR^{a\times b \times c}$.
% \item For any $a,b\in\mathbb{N}$ we define $\IR^{a\times b \times (m\times d)}$ as the set of all objects $(p_{ijk})_{i\in [a],j\in [b], k\in[m]\times[d]}$. In addition, if $a=m$ and $b=d$ we canonically identify an object $A\in \IR^{m\times d \times (m\times d)}$ with the corresponding object from $\IR^{(m\times d) \times (m\times d)}$, which is the set of objects $(p_{ik})_{i,k\in [m]\times [d]}$. Similarly, for all $c\in\mathbb{N}$, we identify objects from $\IR^{m\times d \times a}$ with the corresponding $\IR^{(m\times d) \times a}$ objects without mentioning it. If $A=(p_{ij})\in\mathbb{R}^{a\times (m\times d)}$ and $B=(q_{jk})\in\mathbb{R}^{(m\times d)\times b}$, where $a,b\in\{(m\times d)\}\cup\IN$, we define $A\cdot B=(r_{ik})\in\mathbb{R}^{a\times b}$ via $r_{ik}:=\sum_{j\in[m]\times[d]} p_{ij}q_{jk}$.
%\end{itemize}

For the next result we can either assume that $u$ is a weakly regular decoupling field to an (SLC) problem or that it is a weakly regular Markovian decoupling field to an (MLLC) problem. As usual we denote by $(X,Y,Z)$ the processes appearing in the definition of a decoupling field or a Markovian decoupling field respectively. We also denote by $\sigma^{(i)}$, where $i\in\{1,\ldots,d\}$, the $\mathbb{R}^n$ - valued $i$ - th column of an $\mathbb{R}^{n\times d}$ - valued $\sigma$. The expression $\mathrm{Id}_{m\times d}\in \mathbb{R}^{(m\times d) \times (m \times d)}$ denotes the generalized matrix associated with the identity on $\mathbb{R}^{m\times d}$.

\begin{thm}\label{derivdym} Assume that $\mu,\sigma,f$ are classically differentiable w.r.t.\ $(x,y,z)$ everywhere and assume that either $L_{\sigma,z}=0$ or $n=m=1$ (or both).
Then for almost all initial conditions $x\in\mathbb{R}^n$ there exists a time-continuous version of the process $V$ (which we again denote by $V$) and $d$ square-integrable $\mathbb{R}^{m \times n}$ - valued processes $\tilde{Z}^{(1)},\ldots,\tilde{Z}^{(d)}$, which we can combine to an $\mathbb{R}^{(m\times d) \times n}$ - valued process $\tilde{Z}$, such that
\begin{equation}\label{Vdynam}
V_s=\xi'(X_T)-\sum_{i=1}^d\int_s^T\tilde{Z}^{(i)}_r\mathrm{d} W^{(i)}_r-\int_s^T\varphi(r,V_r,\tilde{Z}_r)\dx r, 
\end{equation}
a.s. for every $s\in[t_1,T]$, where 
$$ \varphi(s,V_s,\tilde{Z}_s):=f_{x,s}+f_{y,s}V_s+f_{z,s} h(s,V_s,\tilde{Z}_s)-$$
$$-V_s\mu_{x,s}-V_s\mu_{y,s}V_s-V_s\mu_{z,s}h(s,V_s,\tilde{Z}_s)-\sum_{i=1}^d\tilde{Z}^{(i)}_s\left(\sigma^{(i)}_{x,s}+\sigma^{(i)}_{y,s}V_s +\sigma^{(i)}_{z,s} h(s,V_s,\tilde{Z}_s)\right), $$
with $f_{x,s}:=f_x(s,X_s,Y_s,Z_s)$, $f_{y,s}:=f_y(s,X_s,Y_s,Z_s)$ etc. and where
$$ h(s,V_s,\tilde{Z}_s):= (\mathrm{Id}_{m\times d}-V_s\sigma_{z,s})^{-1}(V_s\sigma_{x,s}+V_s\sigma_{y,s}V_s+\tilde{Z}_s). $$
\end{thm}

\begin{remark} Firstly, let us explain the meaning of the expression $(\mathrm{Id}_{m\times d}-V_s\sigma_{z,s})^{-1}$ in the definition of $h$. To this end note that $\sigma_{z}(s,X_s,Y_s,Z_s)$ is $\mathbb{R}^{n\times d \times (m \times d)}$ - valued, since $\sigma$ is $\mathbb{R}^{n\times d}$ - and $z$ is $\mathbb{R}^{m \times d}$ - valued. Therefore, the product $V_s\sigma_{z}(s,X_s,Y_s,Z_s)$ is well-defined and assumes values in $\mathbb{R}^{m\times d \times (m\times d)}$. In particular, it can be identified with a quadratic $(m\cdot d) \times (m\cdot d)$ - matrix. We would like to estimate its operator norm (w.r.t.\ the Frobenius norm on $\mathbb{R}^{m\times d}$): To this end consider an arbitrary $\zeta \in\mathbb{R}^{m\times d}$ having the Frobenius norm of $1$. Then $\sigma_{z,s}\zeta$ is $\mathbb{R}^{n\times d}$ - valued and has a Frobenius norm of at most $L_{\sigma,z}$ due to the definition of this constant. Let $v_i\in \mathbb{R}^{n}$, $i=1,\ldots,d$, be the $i$ - th column of this $n\times d$ - matrix. We have $\sum_{i=1}^d|v_i|^2\leq L_{\sigma,z}^2$, where $|\cdot|$ denotes the Euclidean norm. Now $V_s v_i$ is $\mathbb{R}^{m}$ - valued and has a Euclidean norm of at most $L_{u,x}|v_i|$. Therefore, the Frobenius norm of $V_s\sigma_{z,s}\zeta$ is at most $\sqrt{\sum_{i=1}^d L_{u,x} ^2 |v_i|^2}\leq L_{u,x}L_{\sigma,z}<1$. In other words the operator norm of $V_s\sigma_{z,s}$ is at most $L_{u,x}L_{\sigma,z}<1$. We have, thus, shown that
$$ (\mathrm{Id}_{m\times d}-V_s\sigma_{z,s})^{-1}=\sum_{k=0}^\infty \left(V_s\sigma_{z,s}\right)^k=\mathrm{Id}_{m\times d}+\sum_{k=1}^\infty \left(V_s\sigma_{z,s}\right)^k $$
is well-defined and bounded by $\sum_{k=0}^\infty \left(L_{u,x}L_{\sigma,z}\right)^k=(1-L_{u,x}L_{\sigma,z})^{-1}<\infty$ in its operator norm. 
\end{remark}

\begin{proof}[Proof of Theorem \ref{derivdym}]
In order to deduce the dynamics of $V$ we can begin by formally differentiating the forward and the backward equation w.r.t.\ $x\in\mathbb{R}^n$ using strong regularity (see Definitions \ref{regulSLC} and \ref{regulMLLC}). One can verify that one can interchange differentiation and integration and that a chain rule for weak derivatives applies (see Sections A.2 and A.3 in \cite{Fromm2015}). Thus, we obtain that for every version $(\partial_x X, \partial_x Y,\partial_x Z) = (\partial_x X^{t_1,x}, \partial_x Y^{t_1,x},\partial_x Z^{t_1,x})$ of the weak derivative, such that for every $s\in[t_1,T]$ the pair $(\partial_x X_s, \partial_x Y_s)$ is a weak derivative of $(X_s,Y_s)$, we have for every $t\in[t_1,T]$:
\begin{align}\label{dyn xx}
\partial_x X_t = & \mathrm{Id}_n + \int_{t_1}^t \left(\mu_{x,s} \partial_x X_s + \mu_{y,s}\partial_x Y_s +\mu_{z,s}\partial_x  Z_s\right) ds  \\ \nonumber
&+ \sum_{i=1}^d\int_{t_1}^t \left(\sigma^{(i)}_{x,s} \partial_x X_s + \sigma^{(i)}_{y,s}\partial_x Y_s +\sigma^{(i)}_{z,s}\partial_x  Z_s\right)\mathrm{d}W^{(i)}_s
\end{align} 
and
\begin{align}\label{dyn yx}
\partial_x Y_t = \xi'(X_T) \partial_x X_T - \int_{t}^T \left(f_{x,s} \partial_x X_s + f_{y,s}\partial_x Y_s +f_{z,s}\partial_x  Z_s\right) \dx s - \sum_{i=1}^d\int_t^T \partial_x  Z^{(i)}_s\mathrm{d}W^{(i)}_s,
\end{align}
for $\mathbb{P}\otimes\lambda$ - almost all $(\omega,x)\in\Omega\times\IR^n$. Note that $\partial_x  Z^{(i)}$, $i=1,\ldots,d$, are $\mathbb{R}^{m\times n}$ - valued. We denote by $\partial_x  Z$ the corresponding $\mathbb{R}^{(m\times d)\times n}$ - valued process.

By redefining $(\partial_x X, \partial_x Y)$ as the right-hand-sides of \eqref{dyn xx} and \eqref{dyn yx} respectively, we obtain a new pair of processes $(\partial_x X, \partial_x Y)$ that are continuous in time for all $(\omega,x)$ but remain weak derivatives of $X,Y$ w.r.t.\ $x$. From now on, we always assume that $\partial_x X$ and $\partial_x Y$ are continuous in time. We also assume that for fixed $t\in[t_1,T]$ the mappings $\partial_x X_t$ and $\partial_x Y_t$ are weak
derivatives of $X_t$ and $Y_t$ w.r.t.\ $x\in\mathbb{R}^n$. In particular $\partial_x X_{t_1}=\mathrm{Id}_n$ a.s. for almost all $x\in\IR^n$.

Recall that $Y_t=u(t,X_t)$ a.s. for all $(t,x)\in[t_1,T]\times\mathbb{R}^n$. Therefore, for fixed $t\in[t_1,T]$, the weak derivatives of the two sides of the equation w.r.t.\ $x\in\mathbb{R}^n$
must coincide up to an $\Omega\otimes\lambda$ - null set. The chain rule for weak derivatives (see Corollary 3.2 in \cite{Ambrosio1990} or Lemma A.3.1.\ in \cite{Fromm2015}) implies, for any fixed $t\in[t_1,T]$, that we have for 
$\mathbb{P}\otimes\lambda$ - almost all $(\omega,x)$  
\begin{align}\label{chain} \partial_x Y_t\mathbf{1}_{\{\det(\partial_x X_t)\neq 0\}}=u_x(t,X_t)\partial_x X_t\mathbf{1}_{\{\det(\partial_x X_t)\neq 0\}}=V_t\partial_x X_t\mathbf{1}_{\{\det(\partial_x X_t)\neq 0\}}. 
\end{align}

Now, choose a fixed $x\in\R^n$ such that $\partial_x X_{t_1}=\mathrm{Id}_n$ a.s., \eqref{chain}, \eqref{dyn xx}, \eqref{dyn yx} are satisfied for almost all $(t,\omega)\in[t_1,T]\times\Omega$ and, in addition, \eqref{chain} is satisfied for $t=t_1$, $\mathbb{P}$ - almost surely.
Note that, since $\partial_x X$, $\partial_x Y$ are continuous in time, \eqref{dyn xx} and \eqref{dyn yx} in fact hold for all $t\in[t_1,T]$, $\mathbb{P}$ - almost surely.

For arbitrary $k\in\mathbb{N}$ define a stopping time $\tau_k$ via
$$\tau_k:=\inf\left\{t\in[t_1,T]\,\Big|\,\left|(\partial_x X_t)^{-1}\right|\leq k\right\}\wedge T,$$
where $\left|(\partial_x X_t)^{-1}\right|$ denotes the operator norm of $(\partial_x X_t)^{-1}$ if the inverse exists and $\infty$ otherwise.
Note that $\tau_k>t_1$ almost surely for all $k\geq 2$. Also, $\partial_x X$ is an almost surely invertible matrix on $[t_1,\tau_k]$ and we have $V=\partial_x Y\left(\partial_x X\right)^{-1}$ a.e. on this stochastic interval. In particular, $V$, or a version of $V$, is a continuous It\^o process and we can write
$$ V_{s\wedge\tau_k}=V_{\tau_k}-\sum_{i=1}^d\int_{s\wedge\tau_k}^{\tau_k}\tilde{Z}^{(i)}_r\mathrm{d}W^{(i)}_r-\int_{s\wedge\tau_k}^{\tau_k}\varphi_r \dx r, $$
$s\in[t_1,T]$, with processes $\tilde{Z}^{(i)}$ and $\varphi$ that are to be determined. To this end we calculate the dynamics of $V \partial_x X$ using the product rule and compare the result to the dynamics of $\partial_x Y$, which is the same process on the stochastic interval we consider, to obtain an equation that must be satisfied by $\varphi$: Using the product rule we have
$$ \bmdiv^i (V \partial_x X) = \tilde{Z}^{(i)} \partial_x X + V \left(\sigma^{(i)} _{x,\cdot} \partial_x X + \sigma^{(i)} _{y,\cdot}\partial_x Y +\sigma^{(i)} _{z,\cdot}\partial_x  Z\right)=\partial_x  Z^{(i)} $$
and 
$$ \tdiv (V \partial_x X) = \varphi \partial_x X + V \left(\mu_{x,\cdot} \partial_x X + \mu_{y,\cdot}\partial_x Y +\mu_{z,\cdot}\partial_x  Z\right) $$
$$ + \sum_{i=1}^d \tilde{Z}^{(i)} \left(\sigma^{(i)} _{x,\cdot} \partial_x X + \sigma^{(i)} _{y,\cdot}\partial_x Y +\sigma^{(i)} _{z,\cdot}\partial_x  Z\right) = 
f_{x,\cdot} \partial_x X + f_{y,\cdot}\partial_x Y +f_{z,\cdot}\partial_x  Z. $$
We define processes $h^{(i)}$ via $h^{(i)}=\partial_x  Z^{(i)}\left(\partial_x X\right)^{-1}$, where $i=1,\ldots,d$. Note that $h^{(i)}$ are $\mathbb{R}^{m\times n}$ - valued.  We denote by $h$ the corresponding $\mathbb{R}^{(m\times d)\times n}$ - valued process. The above equation for $\bmdiv^i (V \partial_x X)$ yields after multiplication with $\left(\partial_x X\right)^{-1}$:
$$ \tilde{Z}^{(i)} + V \sigma^{(i)} _{x,\cdot} + V\sigma^{(i)} _{y,\cdot}V +V\sigma^{(i)} _{z,\cdot}h=h^{(i)}. $$
By defining $E_i\in \mathbb{R}^{m\times (m\times d)}$ as the generalized matrix such that $E_i z \in \mathbb{R}^{m}$ is the $i$ - th column of an arbitrary $z \in \mathbb{R}^{m\times d}$,
%associated with the linear operator which extracts the $i$ - th column of an arbitrary $\mathbb{R}^{m\times d}$ - matrix, 
we have 
$$ \tilde{Z}^{(i)} + V \sigma^{(i)} _{x,\cdot} + V\sigma^{(i)} _{y,\cdot}V =\left(E_i-V\sigma^{(i)} _{z,\cdot}\right)h. $$
By combining $\sigma^{(i)} _{x,\cdot}$, $\sigma^{(i)} _{y,\cdot}$ and $\sigma^{(i)} _{z,\cdot}$, which are $\IR^{n\times n}$, $\IR^{n\times m}$ and $\IR^{n\times (m\times d)}$ -  valued respectively, over $i=1,\ldots,d$, we obtain the processes $\sigma_{x,\cdot}$, $\sigma_{y,\cdot}$ and $\sigma_{z,\cdot}$, which are $\IR^{n\times d\times n}$, $\IR^{n\times d\times m}$ and $\IR^{n\times d\times (m\times d)}$ -  valued respectively, and we can write
$$ \left(\mathrm{Id}_{m\times d}-V\sigma_{z,\cdot}\right)^{-1}\left(\tilde{Z} + V \sigma _{x,\cdot} + V\sigma_{y,\cdot}V\right) =h=h(\cdot,V,\tilde{Z}), $$
where $\tilde{Z}$ is $\IR^{(m\times d)\times n}$ - valued and where $h$ now denotes both a process and a function by a slight abuse of notation. \\
Furthermore, the equation for $\tdiv (V \partial_x X)$ yields
$$  \varphi + V \left(\mu_{x,\cdot}  + \mu_{y,\cdot} V +\mu_{z,\cdot} h\right) +\sum_{i=1}^d \tilde{Z}^{(i)} \left(\sigma^{(i)} _{x,\cdot} + \sigma^{(i)} _{y,\cdot} V+\sigma^{(i)} _{z,\cdot} h\right) = 
f_{x,\cdot}  + f_{y,\cdot} V +f_{z,\cdot} h $$
or
$$  \varphi = 
f_{x,\cdot}  + f_{y,\cdot} V +f_{z,\cdot} h 
- V \mu_{x,\cdot}  - V\mu_{y,\cdot} V -V\mu_{z,\cdot} h -\sum_{i=1}^d \tilde{Z}^{(i)} \left(\sigma^{(i)} _{x,\cdot} + \sigma^{(i)} _{y,\cdot} V+\sigma^{(i)} _{z,\cdot} h\right). $$

Thus, we have proven that $V$ has the dynamics as in \eqref{Vdynam} but only on the stochastic interval $[t_1,\tau_k]$ for arbitrary $k\in\mathbb{N}$. It follows from definition that $\tau_k$ is non-decreasing in $k$. So, we can define $\tau:=\lim_{k\rightarrow\infty}\tau_k$. Also, we can define $\tilde{Z}$ on the whole of $[t_1,T]$ by setting it to zero outside of $[t_1,\tau)$. It remains to show that $\tilde{Z}$ is square-integrable and that $\tau=T$ holds a.s.:

Observe that on $[t_1,\tau_k]$
$$ \tilde{Z}^{(i)} =\left(E_i-V\sigma^{(i)} _{z,\cdot}\right)\partial_x  Z\left(\partial_x X\right)^{-1} - V \sigma^{(i)} _{x,\cdot} - V\sigma^{(i)} _{y,\cdot}V $$
and that $\partial_x  Z$ is square-integrable due to strong regularity, while $\left(\partial_x X\right)^{-1}$ is uniformly bounded on $[t_1,\tau_k]$. This implies that $\tilde{Z}$ is also square-integrable on $[t_1,\tau_k]$ (with the $L^2$ - norm possibly depending on $k$ at this point). We now distinguish between the cases $L_{\sigma,z}=0$ and $n=m=1$ to show that $\tilde{Z}$ is square-integrable on $[t_1,T]$: \\
In the first case $\sigma^{(i)}_{z,s}$ vanishes and $\varphi$ has in fact a Lipschitz continuous dependence on $\tilde{Z}$. In other words, $\tilde{Z}$, restricted to the interval $[t_1,\tau_k]$, is the control process of a Lipschitz BSDE with terminal condition $V_{\tau_k}$. Since the Lipschitz constant of the FBSDE can be chosen uniformly in $k$, we obtain that $\tilde{Z}$ is square-integrable. \\
Now assume $n=m=1$. Then $V$ is one-dimensional and $\tilde{Z}$ restricted to the interval $[t_1,\tau_k]$ satisfies a quadratic BSDE with terminal condition $V_{\tau_k}$. Using Theorem A.1.11. in \cite{Fromm2015} we have that $\tilde{Z}$, restricted to the interval $[t_1,\tau_k]$, is a BMO - process with a BMO - norm which can be bounded independently of $k$. This yields that $\tilde{Z}$ is a BMO - process. In particular, it is square-integrable.

In both of the two cases above, it is the square-integrability of $\tilde{Z}$ that implies $\tau=T$ a.s.: Using \eqref{dyn xx} we have
$$
\partial_x X_{t\wedge\tau_k} =  \mathrm{Id}_n + \int_{t_1}^{t\wedge\tau_k}\alpha_s\partial_x X_s \dx s  +  \sum_{i=1}^d\int_{t_1}^{t\wedge\tau_k} \beta^{(i)}_s\partial_x X_s\,\mathrm{d}W^{(i)}_s,
$$
a.s. for $t\in[t_1,T]$, where $\alpha_s:=\mu_{x,s} + \mu_{y,s}V_s +\mu_{z,s}h(s,V_s,\tilde{Z}_s)$ and $\beta^{(i)}_s:=\sigma^{(i)}_{x,s} + \sigma^{(i)}_{y,s}V_s +\sigma^{(i)}_{z,s}h(s,V_s,\tilde{Z}_s)$. In other words, $\partial_x X$ satisfies a linear SDE on $[t_1,\tau_k]$, such that we obtain 
$$ \partial_x X_{t\wedge\tau_k}=\exp\left(\int_{t_1}^{t\wedge\tau_k}\alpha_s \dx s+ \sum_{i=1}^d\int_{t_1}^{t\wedge\tau_k} \beta^{(i)}_s\,\mathrm{d}W^{(i)}_s-\sum_{i=1}^d \int_{t_1}^{t\wedge\tau_k} \left(\beta^{(i)}_s\right)^2  \dx s\right).
$$
Note that $\alpha$ and $\beta$ are defined on $[t_1,T]$ (since $\tilde{Z}$ is) and both are square-integrable as $h$ is linear in $\tilde{Z}$. Passing to the limit $k\rightarrow\infty$ and using continuity of $\partial_x X$ in time we have
$$ \partial_x X_{t\wedge\tau}=\exp\left(\int_{t_1}^{t\wedge\tau}\alpha_s \dx s+ \sum_{i=1}^d\int_{t_1}^{t\wedge\tau} \beta^{(i)}_s\dx W^{(i)}_s-\sum_{i=1}^d \int_{t_1}^{t\wedge\tau} \left(\beta^{(i)}_s\right)^2 \dx s\right).
$$
Due to the above, $\partial_x X_{t\wedge\tau}$ is an invertible matrix and, moreover, for almost every fixed $\omega$ the operator norm $\left|(\left(\partial_x X_{t\wedge\tau}\right)(\omega))^{-1}\right|$ can be bounded independently of $t\in[t_1,T]$. In other words, if we fix $\omega\in\Omega$ then for sufficiently large $k\in\mathbb{N}$ the operator norm $\left|(\left(\partial_x X_{t}\right)(\omega))^{-1}\right|$ remains below $k$ for all $t\leq \tau(\omega)$. This implies that, for almost every $\omega$, we have $\tau_k(\omega)=T$ for sufficiently large $k$. Therefore, $\tau(\omega)=T$ for almost all $\omega$.
\end{proof}

\section{Dynamics of higher derivatives}\label{higherD}

Again, assume that $\xi,\mu,\sigma,f$ satisfy (SLC) or (MLLC). For $\mu,\sigma,f$ that are classically differentiable w.r.t.\ $(x,y,z)$ everywhere and $\xi$ that is classically differentiable w.r.t.\ $x\in\mathbb{R}^n$ everywhere we can consider the FBSDE
\begin{align}\label{jointdyn}  X_s & =x+\int_{0}^s\mu(r,X_r,Y_r,Z_r)\dx r+\int_{0}^s\sigma(r,X_r,Y_r,Z_r)\dx W_r , \nonumber\\ 
 Y_s & =\xi(X_T)-\int_{s}^{T}f(r,X_r,Y_r,Z_r)\dx r-\int_{s}^{T}Z_r\dx W_r,  \nonumber\\
 V_s & =\xi'(X_T)-\int_s^T\varphi(r,X_r,Y_r,Z_r,V_r,\tilde{Z}_r)\dx r-\sum_{i=1}^d\int_s^T\tilde{Z}^{(i)}_r\mathrm{d} W^{(i)}_r, \quad s\in[0,T]. 
\end{align}
Here, $\varphi$ is defined as in Theorem \ref{derivdym}. Note that, unlike in Theorem \ref{derivdym}, we do not have to assume $L_{\sigma,z}=0$ or $n=m=1$ for $\varphi$ to be well-defined as long as $\|V\|_\infty< L_{\sigma,z}^{-1}$. \\ Observe that the forward equation in the above system is $n$ - dimensional, while the backward equation is a system of $m+m\cdot n$ equations. The forward process is still $X$, but the backward process is the pair $(Y,V)$.

Now assume that we have a decoupling field to the above problem. We denote by $u$ the first $m$ components of it. Due to Theorem \ref{derivdym} it is natural to assume that the remaining $m\cdot n$ components of the decoupling field are the spatial derivative $u_x$ of $u$. If we assume that $u_x$ is Lipschitz continuous, we can define an $\mathbb{R}^{m\times n\times n}$ - valued process $\check{Y}^{(2)}$ via  $\check{Y}^{(2)}_s=u_{xx}(s,X_s)$. It is natural to ask whether this process satisfies some BSDE similar to the one satisfied by $V$ in Theorem \ref{derivdym} and what the generator of this BSDE would be. Let us heuristically deduce a reasonable candidate for the generator. To this end we can use the same approach as in the proof of  Theorem \ref{derivdym} and view $\check{Y}^{(2)}$  as the product $\partial_x V (\partial_x X)^{-1}$, such that  $\check{Y}^{(2)}\partial_x X$ and $\partial_x V$ coincide. Now assume that $\check{Y}^{(2)}$ is an It\^o process with
$$ \check{Y}^{(2)}_s =\xi''(X_T)-\int_s^T\varphi^{(2)}_r\dx r-\sum_{i=1}^d\int_s^T\check{Z}^{(2,i)}_r\mathrm{d} W^{(i)}_r, $$
where $\varphi^{(2)}_r$ is to be determined. Using formal differentiation of the backward equation \eqref{jointdyn} we obtain the dynamics of $\partial_x V$:
$$\partial_x V_t = \xi''(X_T) \partial_x X_T - \int_{t}^T \left(\varphi_{x,s} \partial_x X_s + \varphi_{y,s}\partial_x Y_s +\varphi_{z,s}\partial_x  Z_s
+\varphi_{v,s}\partial_x  V_s+\varphi_{\tilde{z},s}\partial_x  \tilde{Z}_s\right) \dx s $$
$$ -\sum_{i=1}^d\int_t^T \partial_x \tilde{Z}^{(i)}_s\mathrm{d}W^{(i)}_s.$$
By using the product rule and matching the coefficients we obtain
$$ \bmdiv^i (\check{Y}^{(2)} \partial_x X) = \check{Z}^{(2,i)} \partial_x X + \check{Y}^{(2)} \left(\sigma^{(i)} _{x,\cdot} \partial_x X + \sigma^{(i)} _{y,\cdot}\partial_x Y +\sigma^{(i)} _{z,\cdot}\partial_x  Z\right)=\partial_x  \tilde{Z}^{(i)} $$
and 
$$ \tdiv (\check{Y}^{(2)} \partial_x X) = \varphi^{(2)} \partial_x X + \check{Y}^{(2)} \left(\mu_{x,\cdot} \partial_x X + \mu_{y,\cdot}\partial_x Y +\mu_{z,\cdot}\partial_x  Z\right) $$
$$ + \sum_{i=1}^d \check{Z}^{(2,i)} \left(\sigma^{(i)} _{x,\cdot} \partial_x X + \sigma^{(i)} _{y,\cdot}\partial_x Y +\sigma^{(i)} _{z,\cdot}\partial_x  Z\right)  $$
$$ =\varphi_{x,\cdot} \partial_x X+ \varphi_{y,\cdot}\partial_x Y +\varphi_{z,\cdot}\partial_x  Z
+\varphi_{v,\cdot}\partial_x  V+\varphi_{\tilde{z},\cdot}\partial_x  \tilde{Z}. $$
The dynamics for $\bmdiv^i (\check{Y}^{(2)} \partial_x X)$ yield that 
$$ \partial_x  \tilde{Z}^{(i)}(\partial_x X)^{-1}=\check{Z}^{(2,i)} + \check{Y}^{(2)} \left(\sigma^{(i)} _{x,\cdot}+ \sigma^{(i)} _{y,\cdot}V +\sigma^{(i)} _{z,\cdot}h\right)=:h^{(2,i)}. $$
Combined with the dynamics of $\tdiv (\check{Y}^{(2)} \partial_x X)$ we obtain
\begin{eqnarray*}\varphi^{(2)}&=&\varphi_{x,\cdot}+ \varphi_{y,\cdot}V +\varphi_{z,\cdot}h+\varphi_{v,\cdot}\check{Y}^{(2)}+\varphi_{\tilde{z},\cdot}h^{(2)} \\
&\,&-\check{Y}^{(2)} \left(\mu_{x,\cdot} + \mu_{y,\cdot}V +\mu_{z,\cdot}h\right) \\
&\,&-\sum_{i=1}^d \check{Z}^{(2,i)} \left(\sigma^{(i)} _{x,\cdot} + \sigma^{(i)} _{y,\cdot} V +\sigma^{(i)} _{z,\cdot} h\right).
\end{eqnarray*}

This heuristic calculation can be straightforwardly generalized to higher order derivatives. It is, therefore, natural to make the following definitions:

%Firstly, we can define spaces $\mathbb{R}^{a \times_i n}$ of generalized matrices, where $a\in\{n,m,m\times d\}$ and where $\times_i\, n$ stands for "$\times\, n \times \ldots \times n$", where $\times$ is taken $i$ - times for some $i\in \mathbb{N}_0$:
%\begin{itemize}
%\item If $i=0$ then $\times_i \, n$ is omitted, while for $i=1$ it stands for "$\times n$", which is something we already defined for every $a\in\{n,m,m\times d\}$.
%\item For $i\geq 1$ we define $\mathbb{R}^{a \times_{i+1} n}$ as the linear space of all mappings $[a]\times [n]^i\rightarrow\mathbb{R}$, which is identical to the previous definition for $i=1$.
%\item Now let $A\in\mathbb{R}^{a \times_{i} n}$ for some $i\geq 1$. Then we can canonically identify $A$ with a sequence $A_1,\ldots,A_n$ of  generalized matrices from $\mathbb{R}^{a \times_{i-1} n}$, i.e. mappings $[a]\times [n]^{i-1}\rightarrow\mathbb{R}$, by writing $[n]^i=[n]^{i-1}\times[n]$ and running through the trailing $[n]=\{1,\ldots,n\}$. This identification is useful for various reasons: For instance, we can define a product $C=A\times B\in \mathbb{R}^{a \times_{i} n}$, for arbitrary $B=(B_{jk})\in \mathbb{R}^{n\times n}$ by setting $C_k=\sum_{j=1}^n A_j B_{jk}$, such that $C_1,\ldots,C_k$ form $C$. Similarly, we can define products $A\times B\in \mathbb{R}^{a \times_{i-1} n}$, where $B\in\mathbb{R}^n$.
%\end{itemize}

Assume that $\mu,\sigma,f$ are $k\in\mathbb{N}_1$ times classically differentiable w.r.t.\ $(x,y,z)$ everywhere and that $\xi$ is $k$ times classically differentiable w.r.t.\ $x$ everywhere. 

Define the set $\Theta_k$ recursively via:
$ \Theta_0 := [0,T]\times\Omega\times\mathbb{R}^{n}\times \mathbb{R}^{m}\times\mathbb{R}^{m\times d}, $
$$ \Theta_1 :=\Theta_0\times \left\{ v\in \mathbb{R}^{m \times n}\,\big|\, |v|< L_{\sigma,z}^{-1}\right\}\times\mathbb{R}^{\left(m\times d\right) \times n} $$
and, for $k$ larger than $1$, set $\Theta_{k} := \Theta_{k-1} \times \mathbb{R}^{m \times_k n}\times\mathbb{R}^{\left(m\times d\right) \times_k n}$.
We would like to define functions 
$h^{(k)}:\Theta_k\rightarrow\mathbb{R}^{\left(m\times d\right) \times_k n}$
and
$ \varphi^{(k)}:\Theta_k\rightarrow\mathbb{R}^{m \times_k n} $ generalizing $h$ and $\varphi$.
%where $\times_i\, n$ stands for $\times\, n \times \ldots \times n$, where $\times$ is taken $i$ - times. If $i=0$ then $\times_i \, n$ is omitted, while for $i=1$ it stands for $\times n$. 
For an element
$\theta_k=\left(t,\omega,x,\check{y}^{(0)},\check{z}^{(0)},\ldots,\check{y}^{(k)},\check{z}^{(k)}\right)\in\Theta_k$
define $\theta_{i}:=\left(t,\omega,x,\check{y}^{(0)},\check{z}^{(0)},\ldots,\check{y}^{(i)},\check{z}^{(i)}\right)\in\Theta_i$ for every $i=0,\ldots,k$ and set $\theta:=\theta_0$.
Now, we can define $h^{(1)}$ via
\begin{equation}\label{h1defi} h^{(1)}(\theta_1):=\left(\mathrm{Id}_{m\times d}-\check{y}^{(1)}\sigma_{z}(\theta)\right)^{-1}\left(\check{y}^{(1)}\sigma_{x}(\theta)+\check{y}^{(1)}\sigma_{y}(\theta)\check{y}^{(1)}+\check{z}^{(1)}\right), \end{equation}
 and $h^{(k)}$, for $k\geq 2$, via
$$ h^{(k,j)}(\theta_k):=
\check{z}^{(k,j)} + \check{y}^{(k)} \left(\sigma^{(j)}_{x}(\theta)+ \sigma^{(j)}_{y}(\theta)\check{y}^{(1)} +\sigma^{(j)}_{z}(\theta)h^{(1)}(\theta_1)\right), \quad j=1,\ldots,d. $$
Next we set $\varphi^{(0)}:=f$ and for $k\geq 1$ define $\varphi^{(k)}$ via:
\begin{eqnarray}\label{phirecur}\varphi^{(k)}(\theta_k)&:=&\varphi^{(k-1)}_x(\theta_{k-1})+
\sum_{i=0}^{k-1}\left(\varphi^{(k-1)}_{\check{y}^{(i)}}(\theta_{k-1})\check{y}^{(i+1)}+\varphi^{(k-1)}_{\check{z}^{(i)}}(\theta_{k-1})h^{(i+1)}(\theta_{i+1})\right) \nonumber \\
 &\,&-\check{y}^{(k)} \left(\mu_{x}(\theta) + \mu_{y}(\theta)\check{y}^{(1)}+\mu_{z}(\theta)h^{(1)}(\theta_1)\right) \nonumber \\
 &\,&-\sum_{j=1}^d \check{z}^{(k,j)} \left(\sigma^{(j)} _{x}(\theta) + \sigma^{(j)} _{y}(\theta) \check{y}^{(1)} +\sigma^{(j)} _{z}(\theta) h^{(1)}(\theta_1)\right).
\end{eqnarray}

Note that this definition makes sense only if $\varphi^{(k-1)}$, $\mu$ and $\sigma$ are sufficiently smooth such that all derivatives exist as classical derivatives. We will later derive this from certain differentiability requirements for $f,\mu,\sigma$. 

Also note that $\varphi^{(k-1)}_{\check{y}^{(i)}}(\theta_{k-1})$ is defined as a linear mapping between the spaces $\mathbb{R}^{m \times_i n}$ and $\mathbb{R}^{m \times_{k-1} n}$. Thus, we can merely apply $\varphi^{(k-1)}_{\check{y}^{(i)}}(\theta_{k-1})$ to objects from $\mathbb{R}^{m \times_i n}$, 
while $\check{y}^{(i+1)}$ is from $\mathbb{R}^{m \times_{i+1} n}$. 
However, we can identify $\check{y}^{(i+1)}$ with a vector of $n$ objects from $\mathbb{R}^{m \times_{i} n}$ and apply $\varphi^{(k-1)}_{\check{y}^{(i)}}(\theta_{k-1})$ component-wise. Similarly, we define the product $\varphi^{(k-1)}_{\check{z}^{(i)}}(\theta_{k-1})h^{(i+1)}(\theta_{i+1})$.

Observe that $h^{(k)}$ and $\varphi^{(k)}$ are affine linear in $\check{y}^{(k)},\check{z}^{(k)}$ if $k\geq 2$. Note also that $\varphi^{(1)}=\varphi$ and $h^{(1)}=h$ are, in general, quadratic in $(\check{y}^{(1)},\check{z}^{(1)})$ if $L_{\sigma,z}=0$ and even non-polynomial in $\check{y}^{(1)}$ if $L_{\sigma,z}>0$.

Due to our heuristic considerations, which are straightforward to generalize to arbitrary $k$, it is natural to expect the process $\check{Y}^{(i)}_s = u_i(s,X_s)$, where $u_i$ is the $i$ - the derivative of $u$ w.r.t.\ $x$, assuming it exists, to satisfy
\begin{equation}\label{genYcheck}
\check{Y}^{(i)}_s  =\xi^{(i)}(X_T)-\int_s^T\varphi^{(i)}(r,X_r,\check{Y}^{(0)}_r,\check{Z}^{(0)}_r,\ldots,\check{Y}^{(i)}_r,\check{Z}^{(i)}_r)\dx r-
\sum_{j=1}^d\int_s^T\check{Z}^{(i,j)}_r\mathrm{d} W^{(j)}_r,
\end{equation}
where $\check{Z}^{(i,j)}$ is $\mathbb{R}^{m\times_i n}$ - valued and is extracted from some $\mathbb{R}^{\left(m\times d\right) \times_i n}$ - valued process $\check{Z}^{(i)}$ (again assuming it exists). Here $\xi^{(i)}$ refers to the $i$ - th derivative of $\xi$ w.r.t.\ $x\in\mathbb{R}^n$.

Now, for arbitrary $k\in\mathbb{N}_0$ let us make the following definition:

\begin{defi} We call a function $\left(u^{(0)},\ldots,u^{(k)}\right) \colon [t_0,T] \times \Omega\times \mathbb{R}^n \to \prod_{i=0}^k\mathbb{R}^{m\times_i n}$ a \textit{$k$-decoupling field} for $(\xi,(\mu,\sigma,f))$ if $\left(u^{(0)},\ldots,u^{(k)}\right)(T,\cdot) = (\xi,\ldots,\xi^{(k)})$ a.e. and if for all $t_0 \leq t_1\leq t_2 \leq T$ and all $\mathcal{F}_{t_1}$-measurable random vectors $X_{t_1}\colon \Omega \to \mathbb{R}^n$
there are progressively measurable processes $X\colon [t_{1},t_{2}] \times \Omega \to \mathbb{R}^n$ and
$\check{Y}^{(i)},\check{Z}^{(i)}\colon [t_{1},t_{2}] \times \Omega \to \mathbb{R}^{m \times_i n},\mathbb{R}^{(m \times d)\times_i n}$, where $i=0,\ldots,k$,
such that
\begin{align*}
    X_s &  = X_{t_1} + \int_{t_1}^s \mu(r,X_r,\check{Y}^{(0)}_r,\check{Z}^{(0)}_r) \dx r+\int_{t_1}^s \sigma(r,X_r,\check{Y}^{(0)}_r,\check{Z}^{(0)}_r) \dx W_r, \\
   \check{Y}^{(i)}_s & = \check{Y}^{(i)}_{t_2}-\int_{s}^{t_2}\varphi^{(i)}(r,X_r,\check{Y}^{(0)}_r,\check{Z}^{(0)}_r,\ldots,\check{Y}^{(i)}_r,\check{Z}^{(i)}_r)\dx r-\sum_{j=1}^d\int_{s}^{t_2}\check{Z}^{(i,j)}_r\mathrm{d} W^{(j)}_r, \\
     \check{Y}^{(i)}_s &  = u^{(i)}(s,X_s),
\end{align*}
a.s. for all $s \in [t_1,t_2]$ and all $i=0,\ldots,k$ and such that, in case $k\cdot L_{\sigma,z}>0$, $\|\check{Y}^{(1)}\|_\infty< L_{\sigma,z}^{-1}$. In particular, we assume that all integrals exist and that all $\varphi^{(i)}$, $i=1,\ldots,k$, are well-defined, such that recursion \eqref{phirecur} is satisfied.

We call a $k$-decoupling field a \emph{Markovian $k$-decoupling field} if the processes $X,\check{Y}^{(i)},\check{Z}^{(i)}$ can be chosen in such a way that all $\check{Z}^{(i)}$, $i=0,\ldots,k$, are essentially bounded processes.

We call a $k$-decoupling field \emph{weakly regular} if 
$$L_{u^{(0)},x}<L_{\sigma,z}^{-1},\quad \sup_{s\in[t_0,T]}\|u^{(0)}(s,0)\|_\infty<\infty\quad\textrm{and, in case $k>0$, also} $$
$$ \|u^{(1)}\|_\infty<L_{\sigma,z}^{-1}\quad\textrm{and}\quad\|u^{(i)}\|_\infty+L_{u^{(i)},x}<\infty\quad\textrm{ for all } i=1,\ldots,k. $$
\end{defi}

Note that a standard Markovian decoupling field is a Markovian $0$-decoupling field. Also, note that we do not require $u^{(i)}$ to be the $i$ - th derivative of $u^{(0)}$ at this point. Instead, this is something that needs to be shown under suitable conditions.

In order to be able to construct Markovian $k$-decoupling fields we make Lipschitz continuity requirements for $(\xi,(\mu,\sigma,f))$ which are a straightforward generalization of the more standard MLLC theory for Markovian decoupling fields:

%\begin{definition}
%We say that $(\xi,(\mu,\sigma,f))$ satisfies \emph{$k$-standard Lipschitz conditions} ($k$-SLC) if it satisfies (SLC) and in addition the following is satisfied:
%\begin{itemize}
%\item $\mu,\sigma,f$ are $k$ times classically differentiable w.r.t.\ $(x,y,z)$ everywhere with derivatives that are Lipschitz continuous in $(x,y,z)$.
%\item $\xi$ is $k$ times classically differentiable w.r.t.\ $x$ everywhere with Lipschitz continuous derivatives.
%\end{itemize}
%\end{definition}

\begin{definition}
We say that $(\xi,(\mu,\sigma,f))$ satisfies \emph{$k$-modified local Lipschitz conditions} ($k$-MLLC) if it satisfies (MLLC) and in addition the following is fulfilled:
\begin{itemize}
\item $\mu,\sigma,f$ are $k$ times classically differentiable w.r.t.\ $(x,y,z)$ everywhere with derivatives that are Lipschitz continuous in $(x,y,z)$ on sets of the form $[0,T]\times\mathbb{R}^n\times\mathbb{R}^m\times B$, where $B\subseteq\mathbb{R}^{m\times d}$ is a bounded set. In addition:
\item $\xi$ is $k$ times classically differentiable everywhere with Lipschitz continuous derivatives.
\end{itemize}
\end{definition}
In the above definition $k\in\mathbb{N}_0$. Also ($0$-MLLC) and (MLLC) are the same. We observe that
($k$-MLLC) translates into Lipschitz continuity properties for $\varphi^{(i)}$:
\begin{lemma}\label{phikregul}
Assume that $(\xi,(\mu,\sigma,f))$ satisfy ($k$-MLLC) for some $k\in\mathbb{N}_0$. Then for all $i\in\{0,\ldots,k\}$ the function $\varphi^{(i)}$ is well-defined and $k-i+1$ times weakly differentiable w.r.t.\ 
$$(x,\check{y}^{(0)},\check{z}^{(0)},\ldots,\check{y}^{(i)},\check{z}^{(i)})\in\mathbb{R}^{n}\times\prod_{l=0}^i\mathbb{R}^{m \times_l n}\times\mathbb{R}^{\left(m\times d\right) \times_l n}. $$ Furthermore, these derivatives, which in case $i\geq 1$ includes the function $\varphi^{(i)}$ itself, are essentially bounded on sets of the form $[0,T]\times\mathbb{R}^n\times\mathbb{R}^m\times A\times\prod_{l=1}^i B_l\times C_l$, where $A\subseteq\mathbb{R}^{m\times d}$, $B_l\subseteq\mathbb{R}^{m \times_l n}$ and $C_l\subseteq\mathbb{R}^{\left(m\times d\right) \times_l n}$ are compacts sets, such that, in case $i\geq 1$, the set $B_1$ is contained in the open ball of radius $L_{\sigma,z}^{-1}$.
\end{lemma}
\begin{proof}
We conduct an inductive argument over $i$:

For $i=0$ there is nothing to prove, since $f=\varphi^{(0)}$ satisfies the requirements made in the definition of ($k$-MLLC). Note that a locally Lipschitz continuous function is weakly differentiable with a locally bounded derivative.

Now let us make the general observation that the property of a function of being certain times weakly differentiable w.r.t.\ the specified variables with derivatives bounded on the specified sets is maintained when multiplying two functions which already have this property.
It is, thus, straightforward to deduce from definition \eqref{h1defi} that $h^{(1)}$ is $k$ times weakly differentiable w.r.t.\ the specified variables with the derivatives (including $h^{(1)}$ itself) being bounded on the specified sets. 
Similarly, the functions 
$\theta_1\mapsto\sigma^{(j)}_{x}(\theta)+ \sigma^{(j)}_{y}(\theta)\check{y}^{(1)} +\sigma^{(j)}_{z}(\theta)h^{(1)}(\theta_1)$ and $\theta_1\mapsto\mu_{x}(\theta)+ \mu_{y}(\theta)\check{y}^{(1)} +\mu_{z}(\theta)h^{(1)}(\theta_1)$ are also $k$ times weakly differentiable w.r.t.\ the specified variables with derivatives bounded on the specified sets. The same applies to $h^{(i)}$ for all $i\geq 2$.

Now assume that the statement holds true for some $i\in\{0,\ldots,k-1\}$ and consider the definition of $\varphi^{(i+1)}$ (see \eqref{phirecur}). Note that, since $\varphi^{(i)}$ is $k-i+1$ times weakly differentiable in the above sense with derivatives that are locally bounded in the above sense, the weak derivatives $\varphi^{(i)}_x$, $\varphi^{(i)}_{\check{y}^{(l)}}$, $\varphi^{(i)}_{\check{z}^{(l)}}$, $l=0,\ldots,i$, are $k-i$ times weakly differentiable in the above sense with derivatives locally bounded in the above sense. This directly translates into $\varphi^{(i+1)}$ having the same property due to its definition.
\end{proof}

We define $\theta_{i-}:=(t,\omega,x,\check{y}^{(0)},\check{z}^{(0)},\ldots,\check{y}^{(i-1)},\check{z}^{(i-1)},\check{y}^{(i)})$ such that $\theta_{i}=\left(\theta_{i-},\check{z}^{(i)}\right)$.
$\varphi^{(k)}$ has the following interesting structural property regarding its dependence on the variable $\check{z}^{(k)}$. 

\begin{propo}\label{phiderivP}
Assume that $(\xi,(\mu,\sigma,f))$ satisfy ($k$-MLLC) for some $k\in\mathbb{N}_0$. Then $\varphi^{(k)}_{\check{z}^{(k)}}$ can be expressed as a function of $\theta_1$ only. Moreover, it is merely a function of $\theta_{1-}$ in case $L_{\sigma,z}=0$.
\end{propo}
\begin{proof}
We prove the statement using induction over $k$. For $k=0$ the statement is clearly true as $\varphi^{(0)}=f$. Now let $k\in\mathbb{N}_1$ and assume the statement is true up to $k-1$. Now consider the definition of $\varphi^{(k)}$ (see \eqref{phirecur}). Note that there are only two parts which depend on $\check{z}^{(k)}$: The summand
$$ \varphi^{(k-1)}_{\check{z}^{(k-1)}}(\theta_{k-1})h^{(k)}(\theta_{k}) $$
and the sum
$$ \sum_{j=1}^d \check{z}^{(k,j)} \left(\sigma^{(j)} _{x}(\theta) + \sigma^{(j)} _{y}(\theta) \check{y}^{(1)} +\sigma^{(j)} _{z}(\theta) h^{(1)}(\theta_1)\right). $$
Note that the derivative of the above sum w.r.t.\ $\check{z}^{(k)}$ depends only on the factor $\sigma^{(j)} _{x}(\theta) + \sigma^{(j)} _{y}(\theta) \check{y}^{(1)} +\sigma^{(j)} _{z}(\theta) h^{(1)}(\theta_1)$, which in turn is a function of $\theta_1$ or even just $(\theta,\check{y}^{(1)})=\theta_{1-}$ if $L_{\sigma,z}=0$.

We already know that $\varphi^{(k-1)}_{\check{z}^{(k-1)}}(\theta_{k-1})$ depends on $\theta_1$ only or even just on  $(\theta,\check{y}^{(1)})$ in case $L_{\sigma,z}=0$ (induction hypothesis). Also note that $h^{(k)}(\theta_{k})$ is linear in $\check{z}^{(k)}$:
$$ h^{(k,j)}(\theta_k)=
\check{z}^{(k,j)} + \check{y}^{(k)} \left(\sigma^{(j)}_{x}(\theta)+ \sigma^{(j)}_{y}(\theta)\check{y}^{(1)} +\sigma^{(j)}_{z}(\theta)h^{(1)}(\theta_1)\right), \quad j=1,\ldots,d. $$
Thus, if $L_{\sigma,z}=0$, then $\sigma^{(j)}_{z}$ vanishes and $\varphi^{(k)}_{\check{z}^{(k)}}$ indeed depends on $(\theta,\check{y}^{(1)})$ only.

Now consider the case $L_{\sigma,z}>0$. If $k=1$ there is nothing to proof as all terms depend on $\theta_1$ only. In case $k\geq 2$ the factor $\varphi^{(k-1)}_{\check{z}^{(k-1)}}(\theta_{k-1})$ does not depend on $\check{z}^{(k)}$ and, thus, we only need to apply the derivative w.r.t.\ $\check{z}^{(k)}$ to the factor $h^{(k)}(\theta_k)$ which in fact results in the identity because $k>1$.
\end{proof}

\begin{lemma}\label{phideriv}
 Assume that $(\xi,(\mu,\sigma,f))$ satisfy ($k$-MLLC) for some $k\in\mathbb{N}_0$. If $k\geq 3$ then the weak derivative $\varphi^{(k)}_{\check{z}^{(k-1)}}$ can be expressed as a function of $\theta_{k-}$ only. If $L_{\sigma,z}=0$ the same holds true even for $k=2$.
\end{lemma}
\begin{proof}
Consider the definition of $\varphi^{(k)}$ according to \eqref{phirecur}. After differentiating $\varphi^{(k)}$ w.r.t.\ $\check{z}^{(k-1)}$ all parts which depend on $\check{z}^{(k)}$ disappear including the summand $\varphi^{(k-1)}_{\check{z}^{(i)}}(\theta_{k-1})h^{(i+1)}(\theta_{i+1})$ for $i=k-1$: $\varphi^{(k-1)}_{\check{z}^{(k-1)}}$ does not depend on $\check{z}^{(k-1)}$ according to Proposition \ref{phiderivP} and neither does $h^{(k)}(\theta_{k})$. This is because for $k\geq 3$ the variable $\check{z}^{(k-1)}$ is not contained in $\theta_1$, while for $k\geq 2$ and $L_{\sigma,z}=0$ it is not even contained in $\theta_{1-}$.
 %Observe that $\varphi^{(k-1)}_{\check{z}^{(k-1)}}$ is a function of $\theta_1$ only (Proposition \ref{phiderivP}) and that $\varphi^{(k)}$ is derived from $\varphi^{(k-1)}$ according to \eqref{phirecur}. 
 %Except for the dependence on $\varphi^{(k-1)}$ the only place in the definition of $\varphi^{(k)}$, where an additional $\check{z}^{(k-1)}$ dependence appears is the summand $\varphi^{(k-1)}_{\check{z}^{(i)}}(\theta_{k-1})h^{(i+1)}(\theta_{i+1})$ for $i=k-2$ ($h^{(k-1)}(\theta_{k-1})$ contains $\check{z}^{(k-1)}$).
 %Now if $k\geq 3$ then $\check{z}^{(k-1)}$ is not contained in $\theta_1$. Thus, $\varphi^{(k-1)}$ is linear in $\check{z}^{(k-1)}$. This translates for the most part to $\varphi^{(k)}$: The only place, where an additional $\check{z}^{(k-1)}$ dependence might appear is the summand $\varphi^{(k-1)}_{\check{z}^{(i)}}(\theta_{k-1})h^{(i+1)}(\theta_{i+1})$ for $i=k-2$. In any case, 
 %For the case $L_{\sigma,z}=0$ note that $\check{z}^{(k-1)}$ is not contained $(\theta,\check{y}^{(1)})$ even if $k=2$. Thus, the second case follows from Proposition \ref{phiderivP} as well.
\end{proof}

\section{Main results for higher derivatives}\label{mainR}

Before developing a theory for ($k$-MLLC) problems let us make a few important definitions:

\begin{definition}\label{def:regularity decoupling}
   Let  $\left(u^{(0)},\ldots,u^{(k)}\right) \colon [t_0,T] \times \Omega\times \mathbb{R}^n \to \prod_{i=0}^k\mathbb{R}^{m\times_i n}$ be a weakly regular Markovian $k$-decoupling field for $(\xi,(\mu,\sigma,f))$. We say that it is \emph{strongly regular} if for all fixed $t_1,t_2\in[t_0,T]$, $t_1\leq t_2,$ the processes $X,\check{Y}^{(i)},\check{Z}^{(i)}$, $i=0,\ldots,k$, arising in the definition of a Markovian $k$-decoupling field are a.e unique and satisfy
   \begin{equation}\label{strong_regul1}
      \sup_{s\in [t_1,t_2]}\mathbb{E}_{t_1,\infty}[|X_s|^2]+\sup_{s\in [t_1,t_2]}\mathbb{E}_{t_1,\infty}[|\check{Y}^{(0)}_s|^2]<\infty, %+\mathbb{E}_{t_1,\infty}\left[\int_{t_1}^{t_2}|Z_s|^2 d s\right]<\infty,
   \end{equation}
   for each constant initial value $X_{t_1}=x\in\mathbb{R}$. In addition they are required to be measurable as functions of $(x,s,\omega)$ and even weakly differentiable w.r.t.\ $x\in\mathbb{R}^n$ such that for every $s\in[t_1,t_2]$ the mappings $X_s$, $\check{Y}^{(i)}_s$, $i=0,\ldots,k$ are measurable functions of $(x,\omega)$ and even weakly differentiable w.r.t.\ $x$ such that
   \begin{align}\label{strong_regul2}
     &\esssup_{x\in\mathbb{R}}\sup_{s\in [t_1,t_2]}\mathbb{E}_{t_1,\infty}\left[\left|\partial_x X_s\right|^2\right]<\infty, \nonumber\allowdisplaybreaks\\
     &\esssup_{x\in\mathbb{R}}\sup_{s\in [t_1,t_2]}\mathbb{E}_{t_1,\infty}\left[\left|\partial_x \check{Y}^{(i)}_s\right|^2\right]<\infty, \nonumber\allowdisplaybreaks \\
     &\esssup_{x\in\mathbb{R}}\mathbb{E}_{t_1,\infty}\left[\int_{t_1}^{t_2}\left|\partial_x \check{Z}^{(i)}_s\right|^2\dx s\right]<\infty, \quad i=0,\ldots,k.
   \end{align}
%Furthermore, we say that a $k$-decoupling field on $[t,T]$ is \emph{strongly regular} on a subinterval $[t_1,t_2]\subseteq[t,T]$ if $\left(u^{(0)},\ldots,u^{(k)}\right)$ restricted to $[t_1,t_2]$ is a strongly regular $k$-decoupling field for $\fbsde(u(t_2,\cdot),(\mu,\sigma,f))$.
\end{definition}

\begin{definition}   Let $\left(u^{(0)},\ldots,u^{(k)}\right)\colon [t_0,T] \times \Omega\times \mathbb{R}^n \to \prod_{i=0}^k\mathbb{R}^{m\times_i n}$ be a Markovian $k$-decoupling field for $(\xi,(\mu,\sigma,f))$. We call $\left(u^{(0)},\ldots,u^{(k)}\right)$ \emph{controlled in $z$} if there exists a constant $C>0$ such that for all $[t_1,t_2]\subseteq[t_0,T]$,
$t_1\leq t_2$, and all initial values $X_{t_1}$, the corresponding processes $X,\check{Y}^{(i)},\check{Z}^{(i)}$, $i=0,\ldots,k$, from the definition of a Markovian
$k$-decoupling field satisfy $|\check{Z}^{(i)}_{s}(\omega)|\leq C$, $i=0,\ldots,k$, for almost all $(s,\omega)\in[t_0,T]$. If for a fixed
triple $(t_1,t_2,X_{t_1})$ there are different choices for $X,\check{Y}^{(i)},\check{Z}^{(i)}$, then all of them are supposed to satisfy
the above control. %Furthermore:
%\begin{itemize}
%\item We say that a $k$-decoupling field $\left(u^{(0)},\ldots,u^{(k)}\right)$  on $[t_0,T]$ is controlled in $z$ on a subinterval $[t_1,t_2]\subseteq[t,T]$ if $\left(u^{(0)},u^{(1)}\right)$ restricted to $[t_1,t_2]$ is a Markovian decoupling field for $\left(\left(u^{(0)},u^{(1)}\right)(t_2,\cdot),(\mu,\sigma,f)\right)$ that is controlled
%in $z$.
%\item A $1$-decoupling field $\left(u^{(0)},u^{(1)}\right)$ on an interval $(s,T]$ is said to be controlled in $z$ if it is controlled
%in $z$ on every compact subinterval $[t_0,T]\subseteq(s,T]$ with $C$ possibly depending on $t_0$.
%\end{itemize}
\end{definition}

%For $k\geq 2$ the following definition makes sense and is motivated by our hope that $u^{(i)}$ is bounded by $L_{u^{(i-1)},x}$ for $i=2,\ldots,k$. Note that boundedness of $u^{(1)}$ is already enforced by weak regularity.
%
%\begin{definition}   Let $\left(u^{(0)},\ldots,u^{(k)}\right)\colon [t_0,T] \times \Omega\times \mathbb{R}^n \to \prod_{i=0}^k\mathbb{R}^{m\times_i n}$ be a Markovian $k$-decoupling field for $(\xi,(\mu,\sigma,f))$. We call $\left(u^{(0)},\ldots,u^{(k)}\right)$ controlled in $\check{y}^{(i)}$ if there exists a constant $C>0$ such that for all $[t_1,t_2]\subseteq[t_0,T]$,
%$t_1\leq t_2$, and all initial values $X_{t_1}$, the corresponding processes $X,\check{Y}^{(i)},\check{Z}^{(i)}$, $i=0,\ldots,k$, from the definition of a
%$k$-decoupling field satisfy $|\check{Y}^{(i)}_{s}(\omega)|\leq C$, $i=1,\ldots,k$, for almost all $(s,\omega)\in[t_0,T]$. If for a fixed
%triple $(t_1,t_2,X_{t_1})$ there are different choices for $X,\check{Y}^{(i)},\check{Z}^{(i)}$, then all of them are supposed to satisfy
%the above control.
%\end{definition}

\begin{definition}
%We define the \emph{$k$-maximal interval} $I^k_{\max}\subseteq [0,T]$ as the union of all intervals $[t_0,T]\subseteq [0,T]$ on which a weakly regular $k$-decoupling field exists. \\
For arbitrary $k\in\mathbb{N}_0$ we define the \emph{$k$-maximal interval} $J^{k}_{\max}\subseteq [0,T]$ as the union of all intervals $[t_0,T]\subseteq [0,T]$ on which a weakly regular Markovian $k$-decoupling field exists.
\end{definition}

We have existence, uniqueness and regularity on the maximal interval:

\begin{thm}\label{1-global}
Assume that $(\xi,(\mu,\sigma,f))$ satisfy ($k$-MLLC) for some $k\in\mathbb{N}_1$. \\ Then there exists a unique weakly regular Markovian $k$-decoupling field $\left(u^{(0)},\ldots,u^{(k)}\right)$ on $J^{k}_{\max}$. \\ This Markovian $k$-decoupling field is deterministic, continuous, strongly regular and controlled in $z$. \\
Furthermore, if $J^{k}_{\max}=(s^{k}_{\min},T]$ with some $s^{k}_{\min}\in[0,T)$, then 
$$\sup_{t\in J^{k}_{\max}} L_{u^{(0)}(t,\cdot),x}=L_{\sigma,z}^{-1}\quad\textrm{or}\quad\sup_{t\in J^{k}_{\max}} \|u^{(1)}(t,\cdot)\|_\infty=L_{\sigma,z}^{-1} $$
$$ \textrm{or}\quad\sup_{t\in J^{k}_{\max}} L_{u^{(i)}(t,\cdot),x}=\infty\quad\textrm{or}\quad\sup_{t\in J^{k}_{\max}}  \|u^{(j)}(t,\cdot)\|_\infty=\infty$$ 
for some $(i,j)\in\{1,\ldots,k\}\times\{2,\ldots,k\}$. Otherwise, $J^{k}_{\max}=[0,T]$ holds.
\end{thm}
\begin{proof}
Firstly, consider the following transformation: For a weakly regular Markovian $k$-decoupling field $\left(u^{(0)},\ldots,u^{(k)}\right)$  on an interval $[t_0,T]$ and a parameter $\lambda>0$ we can define $\left(\bar{u}^{(0)},\ldots,\bar{u}^{(k)}\right)$ via 
$$\bar{u}^{(i)}(t,x)=\lambda^{-i}u^{(i)}(t,\lambda^{-1}x),\quad i=0,\ldots,k.$$
It is straightforward to verify that $\left(\bar{u}^{(0)},\ldots,\bar{u}^{(k)}\right)$ is a Markovian $k$-decoupling field to the problem given by $(\bar\xi,(\bar\mu,\bar\sigma,\bar{f}))$, where $\bar\xi(x):=\xi(\lambda^{-1}x)$,
$$ (\bar\mu,\bar\sigma)(t,x,y,z):=\lambda\cdot(\mu,\sigma)(t,\lambda^{-1}x,y,z) \quad\textrm{and}\quad \bar{f}(t,x,y,z):=f(t,\lambda^{-1}x,y,z). $$
Note that the corresponding $h^{(i)}$, $\varphi^{(i)}$ also change: 
$$ \bar{h}^{(i)}(t,x,\check{y}^{(0)},\check{z}^{(0)},\ldots,\check{y}^{(i)},\check{z}^{(i)})=\lambda^{-i}h^{(i)}(t,\lambda^{-1}x,\lambda^{0}\check{y}^{(0)},\lambda^{0}\check{z}^{(0)},\ldots,\lambda^{i}\check{y}^{(i)},\lambda^{i}\check{z}^{(i)}), $$
$$ \bar{\varphi}^{(i)}(t,x,\check{y}^{(0)},\check{z}^{(0)},\ldots,\check{y}^{(i)},\check{z}^{(i)})=\lambda^{-i}\varphi^{(i)}(t,\lambda^{-1}x,\lambda^{0}\check{y}^{(0)},\lambda^{0}\check{z}^{(0)},\ldots,\lambda^{i}\check{y}^{(i)},\lambda^{i}\check{z}^{(i)}). $$
From the transformed problem we can obtain the previous using the same transformation, but with the parameter $\lambda^{-1}$ instead of $\lambda$.
The two problems are essentially equivalent to each other, however, while $L_{\bar{u}^{(0)},x}=\lambda^{-1}L_{u^{(0)},x}$, $L_{\bar\sigma,z}=\lambda L_{\sigma,z}$ and
$L_{\bar{u}^{(0)},x}L_{\bar\sigma,z}=L_{u^{(0)},x}L_{\sigma,z}<1$, we have $L_{\bar{u}^{(i)},x}=\lambda^{-(i+1)} L_{u^{(i)},x}$ for $i\geq 1$. By choosing $\lambda$ sufficiently large we can make sure that $L_{\left(\bar{u}^{(0)},\ldots,\bar{u}^{(k)}\right),x}L_{\bar\sigma,z}<1$ even though $L_{\left(u^{(0)},\ldots,u^{(k)}\right),x}L_{\sigma,z}$ is possibly not below $1$. \\
The size of the $\lambda>0$ necessary to make sure that $1-L_{\left(\bar{u}^{(0)},\ldots,\bar{u}^{(k)}\right),x}L_{\bar\sigma,z}\geq \frac{1}{2}(1-L_{u^{(0)},x}L_{\sigma,z})$ depends on $L_{u^{(0)},x}L_{\sigma,z}<1$ and also the values $L_{u^{(i)},x}<\infty$, $i=1,\ldots,k$, and is monotonically increasing in them. \\
Furthermore, by choosing $\lambda>0$ sufficiently large, we can make sure that $\left(\bar{u}^{(0)},\ldots,\bar{u}^{(k)}\right)$ constructed as above is a weakly regular Markovian decoupling field on $[t_0,T]$ to an (MLLC) problem given by \eqref{genYcheck} with $i=0,\ldots,k$ and $\bar{h}^{(i)}$, $\bar{\varphi}^{(i)}$ instead of $h^{(i)}$, $\varphi^{(i)}$: In order to ensure (MLLC) we must use a passive "inner cutoff" for $\check{Y}^{(1)}$ which truncates every ${n\times m}$ - matrix to a matrix which has an operator norm of at most $c$, where $c>0$ is some constant smaller than $L_{\bar\sigma,z}^{-1}$. Note that such a $c$ exists considering the fact that $\check{Y}^{(1)}$ is bounded by $\|\bar{u}^{(1)}\|_\infty<L_{\bar\sigma,z}^{-1}$. To ensure the type of Lipschitz continuity needed for (MLLC) we can similarly use passive "inner cutoffs" for  $\check{Y}^{(i)}$, $i=2,\ldots,k$, as well, which is possible due to boundedness of $u^{(i)}$, $i=2,\ldots,k$. \\
The above reduction to a more standard setting directly implies uniqueness, strong regularity, continuity and the property of $\left(u^{(0)},\ldots,u^{(k)}\right)$ to be deterministic and controlled in $z$. This is because all these properties are satisfied by $\left(\bar{u}^{(0)},\ldots,\bar{u}^{(k)}\right)$ (see Theorems \ref{UNIqMREGulM}, \ref{GLObalexistM}) and transfer to $\left(u^{(0)},\ldots,u^{(k)}\right)$.

However, the above transformation can also be used to construct weakly regular Markovian $k$-decoupling fields on small intervals in the first place: Let $t\in J^{k}_{\max}$. Our goal is to construct a weakly regular Markovian $k$-decoupling field on a small interval $[t_1,t]$ using $\left(u^{(0)},\ldots,u^{(k)}\right)(t,\cdot)$ as the terminal condition, thereby extending a given weakly regular Markovian $k$-decoupling field to the left. Now choose $\lambda>0$ such that $1-L_{\left(\bar{u}^{(0)},\ldots,\bar{u}^{(k)}\right)(t,\cdot),x}L_{\bar\sigma,z}\geq \frac{1}{2}(1-L_{u^{(0)}(t,\cdot),x}L_{\sigma,z})$.
Also choose an "inner cutoff" for $\check{Y}^{(1)}$ in \eqref{genYcheck} such that the resulting $c>0$ is between $\|\bar{u}^{(1)}(t,\cdot)\|_\infty$ and $L_{\bar\sigma,z}^{-1}$, e.g. $c=\min\{2\|\bar{u}^{(1)}(t,\cdot)\|_\infty,\frac{1}{2}(\|\bar{u}^{(1)}(t,\cdot)\|+L_{\bar\sigma,z}^{-1})\}$. Also choose "inner cutoffs" for $\check{Y}^{(i)}$, $i=2,\ldots,k$, such that a given cutoff is passive if and only if $|\check{Y}^{(i)}|\leq c_i$ with some $c_i\geq \|\bar{u}^{(i)}(t,\cdot)\|_\infty$, which will be specified later.
This results in a higher-dimensional (MLLC) problem, obtained through manipulation of the generators in \eqref{genYcheck}. We can apply Theorem 4.2.17.\ and Remark 4.2.18.\ of \cite{Fromm2015} to it obtaining a weakly regular Markovian decoupling field $\left(\bar{u}^{(0)},\ldots,\bar{u}^{(k)}\right)$ on a small interval $[t_1,t]$. We denote by $\check{Y}^{(i)},\check{Z}^{(i)}$ the solution processes associated with this manipulated problem.\\
Observe that $\check{Z}^{(j)}$, $j=0,\ldots,k$, which are bounded since the Markovian decoupling field is controlled in $z$, can in fact be bounded explicitly using Lemma 2.5.14.\ of \cite{Fromm2015} (or a statement from its proof) by
$$ \left(L_{\left(\bar{u}^{(0)},\ldots,\bar{u}^{(k)}\right),x}\right)\cdot \left\|\bar\sigma\left(\cdot,X,\check{Y}^{(0)},\check{Z}^{(0)}\right)\right\|_\infty,$$ 
where $\sigma(\cdot,X,\check{Y}^{(0)},\check{Z}^{(0)})$ is uniformly bounded since
$$ \|\check{Z}^{(0)}\|_\infty\leq L_{\bar{u}^{(0)},x}\|\bar\sigma(\cdot,\cdot,\cdot,0)\|_\infty\left(1-L_{u^{(0)},x}L_{\sigma,z}\right)^{-1}<\infty, $$
again using Lemma 2.5.14.\ of \cite{Fromm2015} and the fact that (MLLC) is satisfied.\\
%with the bound involving $L_{u^{(0)}(t,\cdot),x}L_{\sigma,z}<1$ and $L_{\left(\bar{u}^{(0)},\ldots,\bar{u}^{(k)}\right)(t,\cdot),x}<\infty$. 
%\red{By making the interval even smaller we can prevent the $\bar{u}^{(1)}$ component of the Markovian decoupling field to go above the chosen cutoff for $\check{Y}^{(1)}$ on the interval $[t_1,t]$.} %In other words, we obtain a Markovian decoupling field on a small interval to the problem given by \eqref{genYcheck}.
Now note that for $i\geq 2$ the backward dynamics of $\check{Y}^{(i)}$ is Lipschitz continuous in $\check{Y}^{(i)},\check{Z}^{(i)}$ with the Lipschitz constant and the offset being uniformly bounded depending on $c,c_2,\ldots,c_{i-1}$ (or just $c$, if $i=2$) and the Lipschitz constants of $\mu,\sigma,f$ and of their derivatives. Since $\|\bar{u}^{(i)}(t,\cdot)\|_\infty<\infty$, $\check{Y}^{(i)}$ remains uniformly bounded depending on $\|\bar{u}^{(i)}(t,\cdot)\|_\infty$ and the aforementioned constants, but independently of its own cutoff $c_i$. This means that we can choose $c_i$ such that the resulting cutoff for $\check{Y}^{(i)}$ is passive and the choice is made in advance depending only on the chosen $c,c_2,\ldots,c_{i-1}$ (which must be chosen before that), the value $\|\bar{u}^{(i)}(t,\cdot)\|_\infty$ and the Lipschitz constants of $\mu,\sigma,f$ and of their derivatives. Ones all constants $c,c_2,\ldots,c_{k}$ are chosen we obtain a weakly regular Markovian decoupling field to the associated (MLLC) problem on some small interval, such that, for $i\geq 2$, the cutoff for $\check{Y}^{(i)}$ is passive. In order to make sure that the cutoff for $\check{Y}^{(1)}$ (via the constant $c>0$) is passive as well, we possibly need to make the interval $[t_1,t]$ somewhat smaller: Using the backward dynamics of $\check{Y}^{(1)}$ this can be estimated based on the bound for $\check{Z}^{(j)}$, $j=0,1$, the constant $c$, the Lipschitz constants for $\mu,\sigma,f$ and size of the terminal value $\|\bar{u}^{(1)}(t,\cdot)\|_\infty<L_{\bar\sigma,z}^{-1}$.  \\
Now due to passiveness of all cuttoffs the resulting Markovian decoupling field for the (MLLC) problem considered above can be directly transformed into a Markovian $k$-decoupling field $\left(u^{(0)},\ldots,u^{(k)}\right)$ on $[t_1,t]$ for the initial ($k$-MLLC) problem. \\
The maximal size of a small interval on which the above construction works can be bounded away from $0$ depending on the values mentioned in Remark 4.2.18.\ of \cite{Fromm2015} (where $\xi$ is to be replaced by $\left(\bar{u}^{(0)},\ldots,\bar{u}^{(k)}\right)(t,\cdot)$ etc.). Essentially, in order to bound the size of the small interval away from zero we merely need a uniform control on $\|u^{(i)}(t,\cdot)\|_\infty+L_{u^{(i)}(t,\cdot),x}$, $i=1,\ldots,k$, and also control $\|u^{(1)}(t,\cdot)\|_\infty$ and $L_{u^{(0)}(t,\cdot),x}$ away from $L_{\sigma,z}^{-1}$.

The above construction rules out the possibility that $J^{k}_{\max}$ is a compact interval different from $[0,T]$. It also rules out the possibility of $J^{k}_{\max}=(s^{k}_{\min},T]$ if $$\sup_{t\in J^{k}_{\max}} \max\left\{\|u^{(1)}(t,\cdot)\|_\infty,L_{u^{(0)}(t,\cdot),x}\right\}<L_{\sigma,z}^{-1}$$
 and 
$$\sup_{t\in J^{k}_{\max}} L_{u^{(i)}(t,\cdot),x}+\|u^{(i)}(t,\cdot)\|_\infty<\infty$$ for all $i=1,\ldots,k$, as in this case we could choose $t$ such that the associated weakly regular Markovian $k$-decoupling field on $[t,T]$ is extended to the left beyond $s^{k}_{\min}$ using the above construction.
\end{proof}

The following key result establishes a connection between the different components of the Markovian $k$-decoupling field $\left(u^{(0)},\ldots,u^{(k)}\right)$.

\begin{thm}\label{roles} 
Assume that $(\xi,(\mu,\sigma,f))$ satisfies ($k$-MLLC) for some $k\in\mathbb{N}_1$.\\
Let $\left(u^{(0)},\ldots,u^{(k)}\right)$ be a weakly regular Markovian $k$-decoupling field for $(\xi,(\mu,\sigma,f))$ on an interval $[t_0,T]$. Then $u:=u^{(0)}$ is a weakly regular (standard) Markovian decoupling field for $(\xi,(\mu,\sigma,f))$ and $u^{(i)}$ is the $i$-th derivative of $u$ w.r.t.\ $x$, for $i=1,\ldots,k$.
\end{thm}
\begin{proof}
Clearly, due to the definition of a weakly regular Markovian $k$-decoupling field, $u$ satisfies the requirements for a weakly regular Markovian decoupling field. Now let $t_1\in [t_0,T]$ be arbitrary. We would like to show that $u^{(i+1)}(t_1,x)=u^{(i)}_x(t_1,x)$ for almost all $x\in\mathbb{R}^n$ for each $i=0,\ldots,k-1$. To this end set $t_2=T$ and consider the processes $X,\check{Y}^{(0)},\check{Z}^{(0)},\ldots,\check{Y}^{(k)},\check{Z}^{(k)}$  associated with the initial condition $X_{t_1}=x\in\mathbb{R}^n$. 
We use notations from Theorem \ref{derivdym} with $Y:=\check{Y}^{(0)}$, $Z:=\check{Z}^{(0)}$, $V:=\check{Y}^{(1)}$, $\tilde{Z}:=\check{Z}^{(1)}$.
Due to strong regularity equations \eqref{dyn xx} and \eqref{dyn yx} are satisfied, while $\partial_x \check{Y}^{(i)}_{t_1}=u^{(i)}_x(t_1,x)$. At the same time we can construct processes $\bar{X}$, $\bar{Y}^{(i)}$ via
$$ \bar{X}_{t}:=\exp\left(\int_{t_1}^{t}\alpha_s ds+ \sum_{j=1}^d\int_{t_1}^{t} \beta^{(j)}_s\mathrm{d}W^{(j)}_s-\sum_{j=1}^d \int_{t_1}^{t} \left(\beta^{(j)}_s\right)^2  ds\right), $$
$$ \bar{Y}^{(i)}_t:=\check{Y}^{(i)}_{t}\bar{X}_{t}\quad\textrm{and}\quad \bar{Z}^{(i,j)}_t:=\check{Z}^{(i,j)}_{t}\bar{X}_{t}+\check{Y}^{(i)}_{t}\beta^{(j)}_t\bar{X}_{t}, \quad\textrm{for }i=1,\ldots,k, $$
where $\alpha_s:=\mu_{x,s} + \mu_{y,s}V_s +\mu_{z,s}h(s,V_s,\tilde{Z}_s)$, $\beta^{(j)}_s:=\sigma^{(j)}_{x,s} + \sigma^{(j)}_{y,s}V_s +\sigma^{(j)}_{z,s}h(s,V_s,\tilde{Z}_s)$. \\
%$V:=\check{Y}^{(1)}$, $\tilde{Z}:=\check{Z}^{(1)}$ using notations from Theorem \ref{derivdym} with $Y:=\check{Y}^{(0)}$, $Z:=\check{Z}^{(0)}$. \\
Note that $\bar{X}$ is square-integrable since $V$, $\tilde{Z}$, $Z$ and, therefore, $\alpha$ and $\beta$ are bounded. More precisely, it is integrable w.r.t.\ any positive power. This implies that $\bar{Y}^{(i)}_t$ and $\bar{Z}^{(i)}_t$ are also integrable w.r.t.\ any power, since $\check{Z}^{(i)}$ and $\check{Y}^{(i)}$ are bounded for $i\geq 1$. \\
%integrable to any power: They are Lipschitz continuous functions of the process $X$, which has at most Lipschitz dynamics in itself due to boundedness of $\tilde{Z}$ and the decoupling condition $Y=u(\cdot,X)$. \\
Now a straightforward application of the It\^o formula yields that $\bar{X}$, $\bar{Y}^{(1)}$, $\bar{Z}^{(1)}$ satisfy the same linear FBSDE as $\partial_x X$, $\partial_x Y$, $\partial_x Z$ do in \eqref{dyn xx} and \eqref{dyn yx}. Note that the coefficients of this linear FBSDE are bounded, the initial condition is $\mathrm{Id}_n$ and the terminal condition is $\partial_x Y_T = \xi'(X_T) \partial_x X_T$ or $\check{Y}^{(1)}_T = \xi'(X_T) \check{X}_T$ respectively. Such linear FBSDE have unique square-integrable solutions assuming the interval $[t_1,T]$ is small enough with the maximal size of the interval depending on the size of the linear coefficients (which are bounded independently of $t_1$) and the value $\|\xi'\|_\infty\|\sigma_{z}\|_\infty\leq L_{u,x}L_{\sigma,z}<1$. 
This implies that $\bar{X}_t=\partial_x X_t$ and $\bar{Y}_t=\partial_x Y_t$ a.s. for all $t\in[t_1,T]$ if $t_1$ is sufficiently close to $T$. In particular, $\bar{Y}^{(1)}_{t_1}=\partial_x Y_{t_1}$ or $V_{t_1}=u^{(0)}_x(t_1,x)$ which means $u^{(1)}(t_1,x)=u^{(0)}_x(t_1,x)$ for almost all $x$. \\
Now note that $\bar{Y}^{(2)},\bar{Z}^{(2)}$ satisfy the same BSDE that is satisfied by $\partial_x \check{Y}^{(1)},\partial_x \check{Z}^{(1)}$. This is because the dynamics of $\bar{Y}^{(2)}$ are implied by those of $\check{Y}^{(2)}$ and $\partial_x X$ and the generator $\varphi^{(k)}$ in \eqref{genYcheck} is chosen precisely in such a way to make sure that the resulting dynamic is as in the BSDE satisfied by $\partial_x \check{Y}^{(1)}$ assuming that the equations $\partial_x \check{Y}^{(0)}=\check{Y}^{(1)}\partial_x X$ and $\partial_x \check{Z}^{(0,j)}=\check{Z}^{(1,j)}\partial_x X+\check{Y}^{(1)}\beta^{(j)} \partial_x X$ are already known. Thus, similar to the above we obtain that $\bar{Y}^{(2)}_{t_1}=\partial_x \check{Y}^{(1)}_{t_1}$ or $u^{(2)}(t_1,x)=u^{(1)}_x(t_1,x)$ for almost all $x$ if $t_1$ is chosen sufficiently close to $T$.

By repeating essentially the same argument finitely many times we obtain that $u^{(i+1)}(t_1,\cdot)=u^{(i)}_x(t_1,\cdot)$ a.e. for all $i=0,\ldots,k-1$ if $[t_1,T]$ is sufficiently small.

We now choose $[t_1,T]$ as large as possible for the above argument to work and set $\tilde{\xi}:=u^{(0)}(t_1,\cdot)$ as a new terminal condition. Note that $u^{(i)}(t_1,\cdot)$ is the $i$ - th derivative of $\tilde{\xi}$ and we are back to the previous setting. This allows us to extend the interval on which $u^{(i+1)}$ and $u^{(i)}_x$ coincide a bit more to the left. After finitely many applications of the same argument we obtain $u^{(i+1)}=u^{(i)}_x$ a.e., $i=0,\ldots,k-1$, on the whole $[t_0,T]$.
\end{proof}

Conversely, we can show:

\begin{propo}\label{roles2} 
Assume that $(\xi,(\mu,\sigma,f))$ satisfies ($k$-MLLC) for some $k\in\mathbb{N}_1$. Let $u$ be a weakly regular (standard) Markovian decoupling field for $(\xi,(\mu,\sigma,f))$ on an interval $[t,T]$, such that $u$ is $k+1$ times weakly differentiable w.r.t.\ $x$ with bounded derivatives and let $u^{(i)}$, $i=0,\ldots,k$, be the version of the $i$ - th spatial derivative which is continuous in $x$. Assume also that $\sup_{s\in[t,T]} \|u^{(1)}(s,\cdot)\|_\infty<L_{\sigma,z}^{-1}$. \\
Then $\left(u^{(0)},\ldots,u^{(k)}\right)$ is a weakly regular Markovian $k$-decoupling field for $(\xi,(\mu,\sigma,f))$ on $[t,T]$. 
\end{propo}
\begin{proof}
 Consider the interval $I:=[t,T]\cap J^{k}_{\max}$. On $I$ we have a weakly regular Markovian $k$-decoupling field the first component of which must be $u$ and the others the spatial derivatives of $u$ up to order $k$ (Theorem \ref{roles}, Theorem \ref{GLObalexistM}). This already shows that $\left(u^{(0)},\ldots,u^{(k)}\right)$ restricted to $I$ is a weakly regular Markovian $k$-decoupling field. According to Theorem \ref{1-global} the interval $I$ is either equal to $[t,T]$ or is open to the left coinciding with $J^{k}_{\max}$. It remains to rule out the second option: In that case either $|u^{(1)}|$ would not be bounded away from $L_{\sigma,z}^{-1}$ or the spatial derivatives $u^{(1)},\ldots,u^{(k)}$ would not be uniformly Lipschitz continuous (Theorem \ref{1-global}) in contradiction to our assumptions.
\end{proof}

The next result (Theorem \ref{tightExp}) shows that the necessary condition for a singularity as formulated in Theorem \ref{1-global} can be significantly simplified. 

\begin{thm}\label{tightExp}
Assume that $(\xi,(\mu,\sigma,f))$ satisfies ($k$-MLLC) for some $k\geq 0$. Let $\left(u^{(0)},\ldots,u^{(k)}\right)$ be the unique weakly regular Markovian $k$-decoupling field on $J^{k}_{\max}$. For $i\in\{1,2\}$ we say that condition (E\,$i$) is satisfied if $k\geq i$ and $\sup_{t\in J^{k}_{\max}} L_{u^{(i)}(t,\cdot),x}=\infty$. Also, we say that (E0) is satisfied if $\sup_{t\in J^{k}_{\max}} L_{u^{(0)}(t,\cdot),x}=L_{\sigma,z}^{-1}$. \\
Assume that  $J^{k}_{\max}=(s^{k}_{\min},T]$ with some $s^{k}_{\min}\in[0,T)$. Then (E0), (E1) or (E2) must be satisfied. 
If, in addition, $L_{\sigma,z}=0$ then (E0) or (E1) must already be satisfied.
\end{thm}
\begin{proof} 
We only consider $k\geq 1$, since for $k=0$ we already have Theorem \ref{EXPlosionM}.
Assume that  $J^{k}_{\max}=(s^{k}_{\min},T]$ with some $s^{k}_{\min}\in[0,T)$.
Note that according to Theorem \ref{roles} we have $L_{u^{(i)}(t,\cdot),x}=\|u^{(i+1)}(t,\cdot)\|_\infty$ for all $i=0,\ldots,k-1$ and all $t\in J^{k}_{\max}$. 
Thus, for $k\leq 2$ and $L_{\sigma,z}>0$ the statement follows directly from Thereom \ref{1-global}. Similarly, there is nothing to prove for $k\leq 1$ even if $L_{\sigma,z}=0$.

Without any loss of generality we assume from now on that either $k\geq 3$ and $L_{\sigma,z}>0$ or that $k\geq 2$ and $L_{\sigma,z}=0$. For both of these two cases we conduct an indirect proof: Let us assume for the case $L_{\sigma,z}>0$ that $\sup_{t\in J^{k}_{\max}} L_{u^{(0)}(t,\cdot),x}<L_{\sigma,z}^{-1}$ and that $\sup_{t\in J^{k}_{\max}} L_{u^{(i)}(t,\cdot),x}<\infty$ for $i=1,2$. For the case $L_{\sigma,z}=0$, however, we assume from now on that $\sup_{t\in J^{k}_{\max}} L_{u^{(i)}(t,\cdot),x}<\infty$ for $i=0,1$. Under these assumptions we want to produce a contradiction to Thereom \ref{1-global}. This means that our objective is to conclude that $\sup_{t\in J^{k}_{\max}} L_{u^{(i)}(t,\cdot),x}<\infty$ for all $i=2,\ldots,k$. To this end it is actually sufficient to show that if $\sup_{t\in J^{k}_{\max}} L_{u^{(0)}(t,\cdot),x}<L_{\sigma,z}^{-1}$ and $\sup_{t\in J^{k}_{\max}} L_{u^{(i)}(t,\cdot),x}<\infty$ for $i=1,\ldots,k-1$, then $\sup_{t\in J^{k}_{\max}} L_{u^{(k)}(t,\cdot),x}<\infty$ as well, since this conclusion would allow an inductive argument over $k$. So, let us assume from now on that $\sup_{t\in J^{k}_{\max}} L_{u^{(0)}(t,\cdot),x}<L_{\sigma,z}^{-1}$ and $\sup_{t\in J^{k}_{\max}} L_{u^{(i)}(t,\cdot),x}<\infty$ for $i=1,\ldots,k-1$.

Let $X_{t_1}=x\in\mathbb{R}^n$ be the initial condition, where $t_1\in (s^{k}_{\min},T]$. Consider FBSDE \eqref{genYcheck}, where $i=0,\ldots,k$, $s\in[t_1,T]$, and where the processes $X,\check{Y}^{(i)},\check{Z}^{(i)}$ are chosen according to the definition of a Markovian $k$-decoupling field. Let us write $Y:=\check{Y}^{(0)}$, $Z:=\check{Z}^{(0)}$, $V:=\check{Y}^{(1)}$ and $\tilde{Z}:=\check{Z}^{(1)}$ for short.
Due to strong regularity (Theorem \ref{1-global}) we can differentiate the forward process $X$ and all processes $\check{Y}^{(i)},\check{Z}^{(i)}$ w.r.t.\ the initial value $x\in\mathbb{R}^n$ and obtain square-integrable processes $\partial_x X$, $\partial_x\check{Y}^{(i)}$, $\partial_x\check{Z}^{(i)}$. We now consider \eqref{genYcheck} for $i=k$:

Note that $\varphi^{(k)}$ is weakly differentiable w.r.t.\ $x$, $\check{y}^{(i)}$ and $\check{z}^{(i)}$, where $i=0,\ldots,k$, with derivatives that are bounded if $\check{z}^{(0)}$, $\check{y}^{(i)}$, $\check{z}^{(i)}$, $i=1,\ldots,k$, are (Lemma \ref{phikregul}). Moreover, notice that $\varphi^{(k)}$ is affine linear in $\check{y}^{(k)}$, $\check{z}^{(k)}$ (see \eqref{phirecur}). The derivatives of $\varphi^{(k)}$ w.r.t.\ $\check{z}^{(0)}$, $\check{y}^{(i)}$, $\check{z}^{(i)}$, $i=1,\ldots,k-1$ are also affine linear in $\check{y}^{(k)},\check{z}^{(k)}$ with coefficients that are bounded based on the bounds for $\check{z}^{(0)}$, $\check{y}^{(i)}$, $\check{z}^{(i)}$, for $i=1,\ldots,k-1$. In addition, the derivatives of $\varphi^{(k)}$ w.r.t.\ $\check{y}^{(k)}$ and $\check{z}^{(k)}$ are bounded based on the bounds for $\check{z}^{(0)}$, $\check{y}^{(i)}$, $\check{z}^{(i)}$, for $i=1,\ldots,k-1$.

Now observe that $\check{Y}^{(i)}$ is bounded by $\sup_{t\in J^{k}_{\max}} L_{u^{(i-1)}(t,\cdot),x}<\infty$ for all $i=1,\ldots,k$. Furthermore, using Lemma 2.5.14.\ in \cite{Fromm2015} (or a statement from its proof) we obtain that for every $i\in\{1,\ldots,k\}$ the process $\check{Z}^{(i)}$ is bounded by
$$ \left(\sup_{t\in [t_1,T]} L_{\left({u}^{(0)},\ldots,{u}^{(i)}\right)(t,\cdot),x}\right)\cdot \left\|\sigma\left(\cdot,X,Y,Z\right)\right\|_\infty,$$ 
where $\sigma\left(\cdot,X,Y,Z\right)$ is uniformly bounded since
$$ \|Z\|_\infty\leq L_{u^{(0)},x}\|\sigma(\cdot,\cdot,\cdot,0)\|_\infty\left(1-L_{u^{(0)},x}L_{\sigma,z}\right)^{-1}<\infty, $$
again using Lemma 2.5.14.\ of \cite{Fromm2015} and the fact that (MLLC) is satisfied.
We know that $\sup_{t\in [t_1,T]}L_{\left({u}^{(0)},\ldots,{u}^{(i)}\right)(t,\cdot),x}$ is bounded uniformly in $t_1$ for $i\leq k-1$, but not yet for $i=k$.
In any case, $\varphi^{(k)}$ is effectively Lipschitz continuous if we use an appropriate "inner cutoff" for the processes $\check{Z}^{(0)},\check{Y}^{(i)},\check{Z}^{(i)}$, $i=1,\ldots,k$. Among these the process $\check{Z}^{(k)}$ is now the only one for which a uniform bound is yet to be established. \\
Differentiating both sides of \eqref{genYcheck} for $i=k$ w.r.t.\ $x$ and using the chain rule of Lemma A.3.2. in \cite{Fromm2015} we obtain
$$
\partial_x\check{Y}^{(k)}_s  =\xi^{(k+1)}(X_T)\partial_x X_T-\sum_{j=1}^d\int_s^T\partial_x\check{Z}^{(k,j)}_r\mathrm{d} W^{(j)}_r $$
$$ -\int_s^T\Delta^{x}_r\partial_x X_r+\sum_{i=0}^k\left(\Delta^{\check{y}^{(i)}}_r \partial_x \check{Y}^{(i)}_r+\Delta^{\check{z}^{(i)}}_r\partial_x \check{Z}^{(i)}_r\right)\dx r, $$
where $\Delta^{x},\Delta^{\check{y}^{(i)}},\Delta^{\check{z}^{(i)}}$, $i=0,\ldots,k$, are some bounded progressively measurable processes, such that $\Delta^{x},\Delta^{\check{y}^{(i)}},\Delta^{\check{z}^{(i)}}$, for $i=0,\ldots,k-1$, are bounded by an expression, which is affine linear in $|\check{Z}^{(k)}|$ with the two coefficients being bounded. The bounds on these coefficients, as well as the bounds on $\Delta^{\check{y}^{(k)}}$, $\Delta^{\check{z}^{(k)}}$, depend on $\sup_{t\in J^{k}_{\max}} L_{u^{(i)}(t,\cdot),x}$ for $i=0,\ldots,k-1$ only. It is of particular interest that unlike some of the other processes the process $\Delta^{\check{z}^{(k-1)}}$ does not depend on $\check{Z}^{(k)}$ as a consequence of Lemma \ref{phideriv}.

According to the proof of Theorem \ref{roles} we can choose $[t_1,T]$ sufficiently small such that $\partial_x X$ is invertible everywhere, $\partial_x \check{Y}^{(i)}(\partial_x X)^{-1}$ is equal to $\check{Y}^{(i+1)}$ for all $i=0,\ldots,k-1$ and $\partial_x \check{Z}^{(i)}(\partial_x X)^{-1}$ is equal to $h^{(i+1)}(\cdot,X,\check{Y}^{(0)},\check{Z}^{(0)},\ldots,\check{Y}^{(i+1)},\check{Z}^{(i+1)})$, which we refer to as $r\mapsto h^{(i+1)}_r$ for short, for all $i=0,\ldots,k-1$. Furthermore, we can define $\check{Y}^{(k+1)}:=\partial_x\check{Y}^{(k)}(\partial_x X)^{-1}$. Clearly, $\check{Y}^{(k+1)}$ is an It\^o process since  $\partial_x\check{Y}^{(k)}$, $(\partial_x X)^{-1}$ are themselves It\^o processes. So we can write
$$
\check{Y}^{(k+1)}_s  =\xi^{(k+1)}(X_T)-\sum_{j=1}^d\int_s^T\check{Z}^{(k+1,j)}_r\mathrm{d} W^{(j)}_r -\int_s^T\gamma_r\dx r, $$
with progressively measurable processes $\check{Z}^{(k+1)}_r$, $\gamma_r$ that are defined by the above expression and are to be determined more explicitly. A straightforward application of the product rule yields that $\check{Z}^{(k+1,j)}$ is given by
$$\check{Z}^{(k+1,j)}_r=\partial_x  \check{Z}^{(k,j)}(\partial_x X_r)^{-1} - \check{Y}^{(k+1)}_r \left(\sigma^{(j)} _{x,r}+ \sigma^{(j)} _{y,r}V_r +\sigma^{(j)} _{z,r}h(r,V_r,\tilde{Z}_r)\right), $$
where we use the notations from Theorem \ref{derivdym}. The product rule is applied to $\check{Y}^{(k+1)}\cdot \partial_x X$. Furthermore, we obtain, again using the product rule and straightforward transformations, that
$$ \gamma_r =\Delta^{x}_r+\left(\sum_{i=0}^{k-2}\left(\Delta^{\check{y}^{(i)}}_r \check{Y}^{(i+1)}_r+\Delta^{\check{z}^{(i)}}_r h^{(i+1)}_r\right)\right)
+\Delta^{\check{y}^{(k-1)}}_r \check{Y}^{(k)}_r+\Delta^{\check{z}^{(k-1)}}_r h^{(k)}_r$$
$$+\Delta^{\check{y}^{(k)}}_r \check{Y}^{(k+1)}_r
+\Delta^{\check{z}^{(k)}}_r\left(\check{Z}^{(k+1)}_r+\check{Y}^{(k+1)}_r \left(\sigma _{x,r}+ \sigma _{y,r}V_r +\sigma_{z,r}h(r,V_r,\tilde{Z}_r)\right)\right)  $$
$$ -\check{Y}^{(k+1)}_r \left(\mu_{x,r} + \mu_{y,r}V_r +\mu_{z,r}h(r,V_r,\tilde{Z}_r)\right)-
\sum_{j=1}^d \check{Z}^{(k+1,j)} \left(\sigma^{(j)} _{x,r} + \sigma^{(j)} _{y,r} V +\sigma^{(j)} _{z,r} h(r,V_r,\tilde{Z}_r)\right). $$
Note that because of our assumptions the processes $\check{Z}^{(0)}$, $\check{Y}^{(i)}$, $\check{Z}^{(i)}$, $\check{Y}^{(k)}$, where $i=1,\ldots,k-1$, are uniformly bounded
and the processes $\check{Y}^{(k+1)}$, $\check{Z}^{(k+1)}$ satisfy a standard Lipschitz BSDE where the Lipschitz constant is bounded uniformly in $t_1$.
The offset process $\Delta^{x}_r+\left(\sum_{i=0}^{k-2}\left(\Delta^{\check{y}^{(i)}}_r \check{Y}^{(i+1)}_r+\Delta^{\check{z}^{(i)}}_r h^{(i+1)}_r\right)\right)
+\Delta^{\check{y}^{(k-1)}}_r \check{Y}^{(k)}_r+\Delta^{\check{z}^{(k-1)}}_r h^{(k)}_r$ of this BSDE, however, can only be bounded by some expression, which is affine linear in $|\check{Z}^{(k)}|$ with coefficients, which are uniformly controlled depending on some constant, which does not depend on $t_1$. Remember that $|\check{Z}^{(k)}|$ is controlled by $1+\sup_{t\in [t_1,T]} L_{u^{(k)}(t,\cdot),x}$ multiplied with some factor. In that it is of the same magnitude as $\|\check{Y}^{(k+1)}\|_\infty$, since $\check{Y}^{(k+1)}_{t}=u^{(k+1)}(t,X_t)$. This allows to obtain a control for $\check{Y}^{(k+1)}_{t_1}=u^{(k+1)}(t_1,x)$, $x\in\mathbb{R}^n$, and, thus, for $L_{u^{(k)}(t_1,\cdot),x}$ using Gronwall's Lemma. In other words, we claim that $\sup_{t\in [t_1,T]} L_{u^{(k)}(t,\cdot),x}$, as a function of $t_1$, grows at most exponentially in time (moving backwards) starting with $L_{\xi^{(k)},x}$ at time $T$: \\
 In order to conduct the above argument rigorously consider the process $\vec{Y}\in\mathbb{R}^{N}$, which is the process $\check{Y}^{(k+1)}$ written into a simple vector, and consider the associated control process $\vec{Z}\in\mathbb{R}^{N\times d}$ obtained from $\check{Z}^{(k+1)}$. Then $\vec{Y}$, $\vec{Z}$ satisfy a linear BSDE with a generator of the form $\eta_s +\kappa_s\vec{Z}_s$, such that there exists a constant $C$ independent of $t_1$ with $|\kappa_s|\leq C$ and $|\eta_s|\leq C(1+\psi_s)$, where $\psi_s:=\sup_{r\in [s,T]} \|u^{(k+1)}(r,\cdot)\|_{2,\infty}$, where $s\in[t_1,T]$ and where we take the essential supremum of the Frobenius norm of the generalized matrix $u^{(k+1)}(r,\cdot)$. Observe that the Euclidean norm of $\vec{Y}_s$ is bounded by $\psi_s$. Using the It\^o formula it is straightforward to deduce the dynamics of the real-valued process $|\vec{Y}|^2$, by which we mean the square of the Euclidean norm. This process satisfies a BSDE with a generator given by $\tdiv(|\vec{Y}|^2)_s=2\vec{Y}_s^\top\left(\eta_s +\kappa_s\vec{Z}_s\right)+|\vec{Z}_s|^2$ and some terminal condition bounded by $\psi^2_T$. Note that in $|\vec{Z}_s|^2$ we apply the Frobenius norm of $\vec{Z}_s$. Using Young's inequality we have $|2Y_s\kappa_s\vec{Z}_s|\leq C_2 \psi^2_s+ |\vec{Z}_s|^2$ with some constant $C_2$ which depends on $C$ only. Thus, by controlling the generator of the BSDE satisfied by $|\vec{Y}|^2$ we obtain
$$\left|u^{(k+1)}(t_1,x)\right|^2_2=|\vec{Y}_{t_1}|^2\leq \psi^2_T+\int_{t_1}^T C_3(1+\psi^2_s)\dx s, $$
with some constant $C_3$, which does not depend on $t_1$.
Since $x\in\mathbb{R}^n$ and $t_1\in (s^{k}_{\min},T]$ are arbitrary, we obtain
$$\psi^2_{t_1}\leq \psi^2_T+\int_{t_1}^T C_3(1+\psi^2_s)\dx s. $$
In other words, $1+\psi^2$ and, therefore, $\psi$ can grow at most exponentially in time (Gronwall's Lemma).
%This means that $\|\check{Y}^{(k+1)}\|_\infty$ grows at most exponentially in time (moving backwards) starting with $\|\xi^{(k+1)}(X_T)\|_\infty=L_{\xi^{(k)},x}$ at time $T$. The same exponential growth translates into $L_{u^{(k)}(t_1,\cdot),x}$, as a function of $t_1$, since $\check{Y}^{(k+1)}_{t_1}=u^{(k+1)}(t_1,x)$, $x\in\mathbb{R}^n$. 
%Even if $[t_1,T]$ must be chosen sufficiently small for the above argument regarding exponential growth to work, we still obtain the same bound on the exponential growth rate on arbitrary intervals $[t,T]$, where $t>s^{k}_{\min}$, by repeating the same argument finitely many times with different terminal conditions moving to the left. %The exponential growth rate remains the same and does not depend on the number of repetitions.
This yields a uniform bound on $L_{u^{(k)}(t,\cdot),x}$, which is independent of $t\in J^{k}_{\max}$. Thus, we achieve the desired contradiction to Theorem \ref{1-global}.
\end{proof}

Let us now formulate and prove three straightforward applications of Theorem \ref{tightExp}.

\begin{corollary} Assume that $(\xi,(\mu,\sigma,f))$ satisfies ($k$-MLLC), for a $k\geq 2$. Then $J^{k}_{\max}=J^{2}_{\max}$.
\end{corollary}
\begin{proof}
If $\left(u^{(0)},\ldots,u^{(k)}\right)$ is a weakly regular Markovian $k$-decoupling field, then $\left(u^{(0)},u^{(1)},u^{(2)}\right)$ is a weakly regular Markovian $2$-decoupling field according to definition. Therefore, $J^{k}_{\max}\subseteq J^{2}_{\max}$. Now assume that $J^{2}_{\max}$ is strictly larger. Then there exists a $t\in J^{2}_{\max}$ such that $J^{k}_{\max}=(t,T]$. Clearly, there exists a weakly regular Markovian $2$-decoupling field on $[t,T]$, which must coincide on $(t,T]$ with the first two components of the unique weakly regular Markovian $k$-decoupling field $\left(u^{(0)},\ldots,u^{(k)}\right)$. However, according to Theorem \ref{tightExp} we have:
$$\sup_{s\in J^{k}_{\max}} L_{u^{(0)}(s,\cdot),x}=L_{\sigma,z}^{-1}\quad\textrm{or}\quad\sup_{s\in J^{k}_{\max}} L_{u^{(1)}(s,\cdot),x}=\infty\quad\textrm{or}\quad\sup_{s\in J^{k}_{\max}} L_{u^{(2)}(s,\cdot),x}=\infty, $$
which contradicts the fact that $\left(u^{(0)},u^{(1)},u^{(2)}\right)$ is the restriction to $(t,T]$ of a weakly regular Markovian $2$-decoupling field on $[t,T]$.
\end{proof}

\begin{corollary}\label{highergood} Assume that $(\xi,(\mu,\sigma,f))$ satisfies ($k$-MLLC) for some $k\in\mathbb{N}_1$ and suppose $L_{\sigma,z}=0$. Then $J^{k}_{\max}=J^{1}_{\max}$.
\end{corollary}
\begin{proof}
If $\left(u^{(0)},\ldots,u^{(k)}\right)$ is a weakly regular Markovian $k$-decoupling field, then $\left(u^{(0)},u^{(1)}\right)$ is a weakly regular Markovian $1$-decoupling field according to definition. Therefore, $J^{k}_{\max}\subseteq J^{1}_{\max}$. Now assume that $J^{1}_{\max}$ is strictly larger. Then there exists a $t\in J^{1}_{\max}$ such that $J^{k}_{\max}=(t,T]$. Clearly, there exists a weakly regular Markovian $1$-decoupling field on $[t,T]$, which must coincide on $(t,T]$ with the first two components of the unique weakly regular Markovian $k$-decoupling field $\left(u^{(0)},\ldots,u^{(k)}\right)$. However, according to Theorem \ref{tightExp} we have:
$$\sup_{s\in J^{k}_{\max}} L_{u^{(0)}(s,\cdot),x}=\infty\quad\textrm{or}\quad\sup_{s\in J^{k}_{\max}} L_{u^{(1)}(s,\cdot),x}=\infty, $$
which contradicts the fact that $\left(u^{(0)},u^{(1)}\right)$ is the restriction to $(t,T]$ of a weakly regular Markovian $1$-decoupling field on $[t,T]$.
\end{proof}

\begin{remark}\label{highergoodRem}

A straightforward consequence of Corollary \ref{highergood} is that if for $(\xi,(\mu,\sigma,f))$ as in the Corollary we have a Markovian decoupling field $u$ on some interval $[t,T]$ such that $u$ is twice weakly differentiable w.r.t.\ $x\in\mathbb{R}^n$ with bounded derivatives then this $u$ is already $(k+1)$ - times weakly differentiable w.r.t.\ $x$  with bounded derivatives, where $k\in\mathbb{N}_1$ is such that $(\xi,(\mu,\sigma,f))$ satisfy ($k$-MLLC):

Indeed, $(u,u_x)$ is a weakly regular Markovian $1$-decoupling field on $[t,T]$ (Proposition \ref{roles2}) and, therefore, $[t,T]\subseteq J^{1}_{\max} = J^{k}_{\max}$. In particular, there exists a weakly regular Markovian $k$-decoupling field on $[t,T]$. According to Theorem \ref{roles} and uniqueness of weakly regular Markovian decoupling fields (Theorem \ref{UNIqMREGulM}) the first component of the Markovian $k$-decoupling field is $u$ and the other components are its spatial derivatives up to order $k$. All derivatives are bounded and Lipschitz continuous in $x$, such that $u$ is $(k+1)$ - times weakly differentiable w.r.t.\ $x$  with bounded derivatives.
\end{remark}

\begin{corollary}\label{criter} Assume that $(\xi,(\mu,\sigma,f))$ satisfies ($k$-MLLC) for some $k\geq 2$ such that $L_{\sigma,z}=0$ and denote by $\left(u^{(0)},\ldots,u^{(k)}\right)$ the unique weakly regular Markovian $k$-decoupling field on $J^{k}_{\max}$. \\
Then $J^{k}_{\max}=[0,T]$ if and only if for all $t\in J^{k}_{\max}$ and all initial values $x\in\mathbb{R}^n$ the corresponding processes $X,\check{Y}^{(i)},\check{Z}^{(i)}$, $i=0,\ldots,k$, on $[t,T]$ associated with $\left(u^{(0)},\ldots,u^{(k)}\right)$ are such that $\check{Y}^{(1)}$ and $\check{Y}^{(2)}$ are uniformly bounded and these bounds can be chosen independently of $(t,x)$.
\end{corollary}
\begin{proof}
Clearly, if $J^{k}_{\max}=[0,T]$ then $\check{Y}^{(1)},\check{Y}^{(2)}$ are uniformly bounded in the above sense due to the decoupling condition and the definition of weak regularity.

Now assume that $J^{k}_{\max}=(s^{k}_{\min},T]$ with some $s^{k}_{\min}\in[0,T)$. Let $t\in J^{k}_{\max}$ and $x\in\mathbb{R}^n$ be arbitrary. The corresponding $\check{Y}^{(1)},\check{Y}^{(2)}$ satisfy $\check{Y}^{(i)}_t=u^{(i-1)}_x(t,x)$, $i=1,2$, due to Theorem \ref{roles}. Therefore, $L_{u^{(i)}(t,\cdot),x}=\|u^{(i)}_x(t,\cdot)\|_\infty$, $i=0,1$, are uniformly bounded independently of $t$, if we assume that $\check{Y}^{(1)}$, $\check{Y}^{(2)}$ are uniformly bounded in the aforementioned sense. Thus, under this assumption we obtain a contradiction to Theorem \ref{tightExp}.
\end{proof}

\section{An illustrating example}\label{examp}

We consider the following (MLLC) problem: The dimensions are $n=2$, $m=d=1$. The functions $\xi,\mu,\sigma,f$ are such that
$$ \mu(t,x,y,z)=\left(0,z^2\right)^\top,\qquad \sigma(t,x,y,z)=\left(1,0\right)^\top, $$
$$ f(t,x,y,z)=0,\qquad\xi(x)=g\left(x^{(1)}\right)-\delta\left(x^{(2)}\right), $$
where $g, \delta:\IR\rightarrow\IR$ are deterministic and Lipschitz-continuous functions, such that $g$ is non-decreasing. This system was studied in \cite{Proemel2015} and its decoupling field was used to construct solutions to the Skorokhod embedding problem for a class of Gaussian processes. The existence of a Markovian decoupling field was shown in Lemma 4.1.\ of \cite{Proemel2015} in a rather straightforward application of the method of decoupling fields. The remainder of that work was dedicated to showing that this decoupling field is sufficiently smooth. This smoothness is needed for constructing strong solutions to the Skorokhod embedding problem.

Now let $\kappa\in\mathbb{N}$, $\kappa\geq 2$.
Our objective is to show that if $(\xi,(\mu,\sigma,f))$ satisfies ($\kappa$-MLLC), then there exists a unique weakly regular Markovian $\kappa$-decoupling field on $[0,T]$. For $\kappa=2$ this was already proven in \cite{Proemel2015} and this case is sufficient for the purposes of that work. However, the argumentation was rather elaborate and we want to conduct a more compact proof using the results of section \ref{mainR}. Such a more compact argument opens up the door for studying more complex systems, in particular the FBSDE appearing in the study of the Skorokhod embedding problem for more general diffusions.

Now according to Theorem \ref{1-global} a unique weakly regular Markovian $\kappa$-decoupling field $\left(u^{(0)},\ldots,u^{(\kappa)}\right)$ does exist on the interval $J^{\kappa}_{\max}$ and according to Theorem \ref{roles} the first component $u^{(0)}$ is the actual decoupling field and the other components are its spatial derivatives. Thus, it remains to show that $J^{\kappa}_{\max}=[0,T]$ holds. It is natural to use the criterion of Corollary \ref{criter} to this end: \\
Let us write down the backward dynamics of the processes $\check{Y}^{(1)}$ and $\check{Y}^{(2)}$ explicitly. For this purpose we essentially need to calculate $\varphi^{(1)}$ and $\varphi^{(2)}$. Note that $\check{Y}^{(1)}$ and $\varphi^{(1)}$ are $\mathbb{R}^{1\times 2}$ - valued, while $\check{Y}^{(2)}$ and $\varphi^{(2)}$ are $\mathbb{R}^{1\times 2\times 2}$ - valued. Now recall the recursive definition of $\varphi^{(k)}$:
$$ \varphi^{(k)}(\theta_k):=\varphi^{(k-1)}_x(\theta_{k-1})+
\sum_{i=0}^{k-1}\left(\varphi^{(k-1)}_{\check{y}^{(i)}}(\theta_{k-1})\check{y}^{(i+1)}+\varphi^{(k-1)}_{\check{z}^{(i)}}(\theta_{k-1})h^{(i+1)}(\theta_{i+1})\right) $$
$$ -\check{y}^{(k)} \left(\mu_{x}(\theta) + \mu_{y}(\theta)\check{y}^{(1)}+\mu_{z}(\theta)h^{(1)}(\theta_1)\right)-
\sum_{j=1}^d \check{z}^{(k,j)} \left(\sigma^{(j)} _{x}(\theta) + \sigma^{(j)} _{y}(\theta) \check{y}^{(1)} +\sigma^{(j)} _{z}(\theta) h^{(1)}(\theta_1)\right). $$
In our case $d=1$ so we write $\check{z}^{(k)}$ instead of $\check{z}^{(k,1)}$, by a slight abuse of notation. Now for $k=1$ we obtain
$$ \varphi^{(1)}(\theta_1)=f_x(\theta)+
\left(f_{\check{y}^{(0)}}(\theta)\check{y}^{(1)}+f_{\check{z}^{(0)}}(\theta)h^{(1)}(\theta_{1})\right) $$
$$ -\check{y}^{(1)} \left(\mu_{x}(\theta) + \mu_{y}(\theta)\check{y}^{(1)}+\mu_{z}(\theta)h^{(1)}(\theta_1)\right)-
\check{z}^{(k)} \left(\sigma_{x}(\theta) + \sigma_{y}(\theta) \check{y}^{(1)} +\sigma_{z}(\theta) h^{(1)}(\theta_1)\right), $$
where we used the definitions $\varphi^{(0)}=f$ and $\theta=\theta_{0}$. Due to the fact that $\sigma$ is a constant, we obtain 
$$ h^{(k)}(\theta_k)=\check{z}^{(k)}, $$
for all $k\geq 1$. An additional simplification occurs when we exploit $f=0$ and the fact that $\mu$ depends only on $\check{z}^{(0)}$: The definition of $\varphi^{(1)}(\theta_1)$ simplifies to 
$$ \varphi^{(1)}(\theta_1)= -\check{y}^{(1)} \left(\mu_{z}(\theta)h^{(1)}(\theta_1)\right) = -2\check{y}^{(1)}\left(0,\check{z}^{(0)}\right)^\top\check{z}^{(1)}. $$
Note that $\check{z}^{(1)}$ is actually $\mathbb{R}^{\left(1\times 1\right)\times 2}$ - valued, but we view it as simply $\mathbb{R}^{1\times 2}$ - valued, again by a slight abuse of notation. \\
Next we look at  $\varphi^{(2)}$ which describes the dynamics of $\check{Y}^{(2)}$:
$$ \varphi^{(2)}(\theta_2)=\left(\varphi^{(1)}_{\check{z}^{(0)}}(\theta_{1})\check{z}^{(1)}+\varphi^{(1)}_{\check{y}^{(1)}}(\theta_{1})\check{y}^{(2)}+\varphi^{(1)}_{\check{z}^{(1)}}(\theta_{1})\check{z}^{(2)}\right)-\check{y}^{(2)} \left(\mu_{z}(\theta)h^{(1)}(\theta_1)\right). $$
Note that $\varphi^{(1)}$ is linear in $\check{y}^{(1)}$ and linear in $\check{z}^{(1)}$. Thus, we obtain for all test vectors $v\in\mathbb{R}^2$:
$$ \varphi^{(2)}(\theta_2)v=-2\check{y}^{(1)}\left(0,\check{z}^{(1)}v\right)^\top\check{z}^{(1)}-2\left(\check{y}^{(2)}v\right)\left(0,\check{z}^{(0)}\right)^\top\check{z}^{(1)} $$
$$ -2\check{y}^{(1)}\left(0,\check{z}^{(0)}\right)^\top\check{z}^{(2)}v-2\check{y}^{(2)} \left(0,\check{z}^{(0)}\right)^\top\check{z}^{(1)}v. $$

Now let $t\in J^{\kappa}_{\max}$ and $x\in\mathbb{R}^2$ be arbitrary and consider the corresponding processes $X,\check{Y}^{(i)},\check{Z}^{(i)}$, $i=0,\ldots,\kappa$, on $[t,T]$. Note that $\check{Y}^{(1)},\check{Z}^{(1)}$ satisfy a backward SDE given by the generator $\varphi^{(1)}$. Now introduce a Brownian motion with drift $\tilde{W}$ by adding a drift given by the density $-2\check{Y}^{(1)}\left(0,\check{Z}^{(0)}\right)^\top$ to the $1$ - dimensional Brownian motion $W$. $\tilde{W}$ is in fact a Brownian motion under an equivalent probability measure $\mathbb{Q}$ (Girsanov's theorem). Under this measure $\check{Y}^{(1)}$ is a martingale. Thus, $\check{Y}^{(1)}$ is uniformly bounded by the Lipschitz constant of $\xi$. This already implies uniform boundedness of $u^{(1)}=u^{(0)}_x$.
Using Lemma 2.5.14.\ in \cite{Fromm2015} we also obtain uniform boundedness of $\check{Z}^{(0)}$. Also, using Theorem A.5.\ of \cite{Proemel2015} we obtain that $\check{Z}^{(1)}$ is a BMO-process w.r.t.\ $\mathbb{Q}$ and, therefore, w.r.t.\ $\mathbb{P}$. Both BMO-norms can be controlled independently of our choice for $t$ and $x$, due to uniform boundedness of $-2\check{Y}^{(1)}\left(0,\check{Z}^{(0)}\right)^\top$. Moreover, a straightforward adaptation of Lemma 4.8.\ of \cite{Proemel2015} to the interval  $[t,T]\subseteq J^{\kappa}_{\max}$ provides that $\check{Z}^{(1)}=\frac{\dx}{\dx x^{(1)}} u^{(1)}(\cdot, X)$, for which we use in particular that $\left(u^{(0)},\ldots,u^{(\kappa)}\right)$ is deterministic and continuous (Theorem \ref{1-global}). Using Theorem \ref{roles} we obtain that $\check{Z}^{(1)}=\frac{\dx}{\dx x^{(1)}} u^{(1)}(\cdot, X)$ is in fact a.e. equal to $\check{Y}^{(2)}(1,0)^\top$, which is linear in $\check{Y}^{(2)}$.

Next we look at the dynamics of $\check{Y}^{(2)}v$: If we perform the same measure change as above, we do not necessarily obtain a martingale; it is merely the term $-2\check{Y}^{(1)}\left(0,\check{Z}^{(0)}\right)^\top\check{Z}^{(2)}v$ which disappears as a consequence of the change. However, we still obtain uniform boundedness of the process $\check{Y}^{(2)}$ using Lemma A.4. of \cite{Proemel2015}. The applicability of this Lemma follows from boundedness of the processes $\check{Y}^{(1)}$, $\check{Z}^{(0)}$, the BMO-property of $\check{Z}^{(1)}$ and the fact that in the term $-2\check{Y}^{(1)}\left(0,\check{Z}^{(1)}v\right)^\top\check{Z}^{(1)}$ one of the $\check{Z}^{(1)}$ can be replaced by $\check{Y}^{(2)}(1,0)^\top$. 

Now Corollary \ref{criter} implies that $J^{\kappa}_{\max}=[0,T]$ for the ($\kappa$-MLLC) problem considered above.

{\small
\bibliography{quellen}{}
\bibliographystyle{amsalpha}
}

\end{document}